\documentclass[12pt]{amsart}

\usepackage[a4paper,margin=25mm,top=30mm]{geometry}

\usepackage{amsfonts, amsmath, amssymb, amsgen, amsthm, amscd}
\usepackage{mathtools}
\usepackage{txfonts}
\usepackage[utf8]{inputenc} 

\usepackage{color}

\usepackage[all]{xy}

\usepackage{hyperref}
\usepackage{mathtools,slashed}

\usepackage{enumerate}

\def\a{\alpha}
\def\b{\beta}
\def\d{\delta}
\def\D{\Delta}
\def\g{\gamma}

\def\s{\sigma}

\def\t{\theta}

\def\ve{\varepsilon}
\def\vp{\varphi}

\def\p{\partial}

\def\ot{\otimes}
\def\odots{\ot\cdots\ot}

\def\nb{\nabla}

\def\Db{\blacktriangledown}

\def\rt{\triangleright}
\def\lt{\triangleleft}

\def\lbiprod{{>\!\!\!\triangleleft\kern-.33em\cdot\, }}
\def\rbiprod{{\cdot\kern-.33em\triangleright\!\!\!<}}

\def\D{\Delta}

\def\Ad{\mathop{\rm Ad}\nolimits}

\newcommand{\ps}[1]{~\hspace{-4pt}_{^{(#1)}}}

\newcommand{\ns}[1]{~\hspace{-4pt}_{_{{<#1>}}}}

\newcommand{\nsb}[1]{~\hspace{-4pt}_{^{[#1]}}}

\usepackage{enumerate}

\newcommand{\G}[1]{\mathfrak{#1}}
\newcommand{\C}[1]{\mathcal{#1}}
\newcommand{\B}[1]{\mathbb{#1}}

\newcommand{\Hom}{{\rm Hom}}

\newcommand{\tr}{\triangleright}

\renewcommand{\leq}{\leqslant}
\renewcommand{\geq}{\geqslant}

\numberwithin{equation}{section}

\newtheorem{theorem}{Theorem}[section]
\newtheorem{proposition}[theorem]{Proposition}

\newtheorem{corollary}[theorem]{Corollary}
\theoremstyle{definition}

\newtheorem{remark}[theorem]{Remark}

\title{Quantum van Est Isomorphism}

\author{A. Kaygun}
\address{Istanbul Technical University, Istanbul, Turkey}
\email{kaygun@itu.edu.tr}

\author{S. Sütlü}
\address{Gebze Technical University, Kocaeli, Turkey}
\email{serkansutlu@gtu.edu.tr}

\begin{document}

\begin{abstract}
  Motivated by the fact that the Hopf-cyclic (co)homologies of function algebras over Lie groups and  universal enveloping algebras over Lie algebras capture the Lie group and Lie algebra (co)homologies, we hereby upgrade the classical van Est isomorphism to ones between the Hopf-cyclic (co)homologies of quantized algebras of functions and quantized universal enveloping algebras, both in $h$-adic and $q$-deformation frameworks.
\end{abstract}

\maketitle
%

\section*{Introduction}

\subsection*{The van Est Isomorphisms}

Given a connected semisimple Lie group $G$ and a maximal compact subgroup $K$, and their Lie
algebras $\G{g}$ and $\G{k}$, in~\cite{vanEst53} van Est proved that there is an isomorphism
between the continuous Lie group cohomology of $G$, and the Lie algebra cohomology of $\G{g}$
relative to $\G{k}$. In this paper, we construct three different versions of the van Est
isomorphisms in Hopf-cyclic (co)homology using completely different techniques:

\begin{enumerate}[(i)]
\item For the Hopf algebra $\C{O}(G)$ of the regular functions over an algebraic group $G$, or coordinate algebra of a simple matrix Lie group, and
  the enveloping Hopf algebra $U(\G{g})$ of the Lie algebra $\G{g}$ of $G$.
\item For the $h$-adic Hopf algebra of regular functions $\C{O}_h(G)$ of a
  Poisson Lie group $G$, and the $h$-adic enveloping algebra $U_h(\G{g})$ of the Lie
  algebra $\G{g}$ of $G$.
\item For the coordinate Hopf algebra of the $q$-deformation $\C{O}_q(G)$ of a matrix Lie group $G$, and the Drinfeld-Jimbo enveloping algebra $U_q(\G{g})$ of the Lie algebra $\G{g}$ of $G$.
\end{enumerate}

It has been long established that certain classes of Hopf algebras may stand for the quantum
analogues of both Lie groups and Lie algebras. Then in \cite[Prop. 7]{ConnMosc98}, see also \cite[Thm. 15]{ConnMosc}, Connes and
Moscovici showed that the Hopf-cyclic cohomology $HC^*(U(\G{g}),M)$ of the universal
enveloping algebra of a Lie algebra $\G{g}$ with coefficients in a $\G{g}$-module $M$, is
the Lie algebra homology $H_*(\G{g},M)$ of $\G{g}$ with coefficients in the same
$\G{g}$-module $M$. Also proved therein was the analogous isomorphism between the
Hopf-cyclic homology and the Lie algebra
cohomology. 

Once the Hopf-cyclic homology and cohomology are identified as the correct cohomological
context, we first reconstruct the classical van Est isomorphism for the Hopf-cyclic
(co)homology of Lie groups and their Lie algebras. Allowing the full
module-comodule generality on the coefficient space, we obtain the van Est type isomorphisms
\[
HC_\ast(U(\G{g}),U(\G{k}),M^\vee ) \cong HC_\ast(\C{O}(G/K),M^\vee )
\]
in Proposition \ref{prop-vanEst-classical-I}, and
\[
HC^\ast(U(\G{g}),U(\G{k}),M) \cong HC^\ast(\C{O}(G/K),M)
\]
in Proposition \ref{prop-vanEst-classical-II}. More precisely, we identify in the former the Hopf-cyclic homology of the $\C{O}(G)$-comodule algebra $\C{O}(G/K) \subseteq \C{O}(G)$ with coefficients in the SAYD module $M^\vee:=\Hom(M,\B{C})$ over $\C{O}(G)$, with the Hopf-cyclic homology of $U(\G{g})$ relative to
$U(\G{k})$ of the Lie algebra $\G{k}$ of a maximal compact subgroup $K\subseteq G$ with coefficients in the SAYD contra-module $M^\vee$ over $U(\G{g})$. In the latter, on the other hand, we establish an isomorphism between the Hopf-cyclic cohomology of the $\C{O}(G)$-comodule algebra $\C{O}(G/K) \subseteq \C{O}(G)$ with coefficients in the SAYD contra-module $M$ over $\C{O}(G)$, with the Hopf-cyclic homology of $U(\G{g})$ relative to $U(\G{k})$ with coefficients in the SAYD module $M$ over $U(\G{g})$.


Next, we obtain the corresponding van Est isomorphisms for the $h$-adic universal enveloping
algebras of Lie bialgebras and the $h$-adic coordinate algebras of their Poisson-Lie
groups. More precisely, in Theorem \ref{vanEst-I-cohom} and Theorem \ref{vanEst-II-cohom}, we obtain isomorphisms
\[
HC^\ast(U_h(\G{g}),U_h(\G{k}),M) \cong HC^\ast(\C{O}_h(G),M)
\]
and
\[
HC_\ast(U_h(\G{g}),U_h(\G{k}),M^\vee ) \cong HC_\ast(\C{O}_h(G),M^\vee ).
\]
To this end, it sufficed to use the natural $h$-filtration of the complexes to reduce the statement to the classical van Est isomorphism between Lie group (co)homology and (relative) Lie algebra (co)homology.


The $q$-deformation case, however, proved to be substantially different than its $h$-adic
counterpart. In the absence of a natural filtration, we were forced to rely on the dual
pairings between $U_q(g\ell_n)$ and $\C{O}_q(GL(n))$, $U_q^{\text{ext}}(s\ell_n)$ and $\C{O}_q(SL(n))$, $U_{q^{1/2}}(so_{2n+1})$ and $\C{O}_q(SO(2n+1))$, $U_q^{\text{ext}}(so_{2n})$ and $\C{O}_q(SO(2n))$, and finally $U_q^{\text{ext}}(sp_{2n})$ and $\C{O}_q(Sp(2n))$. Then, for each of these pairs $U_q(\G{g})-\C{O}_q(G)$, we obtained in
Theorem \ref{thm-q-adic-vanEst-I} the isomorphism
\[
HC^*(U_q(\G{g}), U_q(\G{k}),M) \cong HC^*(\C{O}_q(G/K),M)
\]
of Hopf-cyclic cohomologies, and in Theorem \ref{thm-q-adic-vanEst-II} the isomorphism
\[
HC_*(U_q(\G{g}), U_q(\G{k}),M^\vee ) \cong HC_*(\C{O}_q(G/K),M^\vee )
\]
of Hopf-cyclic homologies.

\subsection*{The Janus map}

In~\cite{KaygunSutlu:CharMap} we obtained a cup-product like pairing
\[ HC^p_H(H,M)\otimes HC^q_H(A,M)\to HC^{p+q}(A) \] whose ingredients were the Hopf-cyclic
cohomology of a Hopf algebra $H$ with coefficients in a SAYD module $M$, the Hopf-cyclic
cohomology of an $H$-module algebra $A$ with the same coefficient module $M$, and the
ordinary cyclic cohomology of $A$. Even though the pairing is most useful in its
cohomological manifestation it was constructed on the level of (co)cyclic objects:
\begin{equation}
  \label{eq:cm-char-map}
  diag_{\Delta C}(C_\bullet^H(H,M)\otimes C^\bullet_H(A,M))\to C^\bullet(A).
\end{equation}
One of the most interesting properties of the cyclic category $\Delta C$ is that it is
isomorphic to its opposite category $\Delta C^{op}$. For ordinary algebras and coalgebras,
the default (co)cyclic object $C^\bullet$ is possibly non-trivial while its cyclic dual
${}^\circ C^\bullet$ is surely trivial (point-like).  However, this is not necessarily the
case for Hopf algebras and Hopf-equivariant (co)cyclic objects we construct for Hopf-module
(co)algebras. Thus, one can think of the Connes-Moscovici pairing~\eqref{eq:cm-char-map} in
the dual cyclic setting where the target ${}^\circ C^\bullet(A)$ collapses. This fact allows
us to reformulate~\eqref{eq:cm-char-map} as a duality between graded vector spaces of dual
cyclic (co)homologies ${}^\circ HC^*_H(H,M)$ and ${}^\circ HC^*_A(A,M)$.  Thus the van Est
map and the Connes-Moscovici characteristic map become the two different faces of the same
pairing which we now call \emph{the Janus map}.  We investigate this approach in the
Appendix.

\subsection*{Plan of the article} 

The paper is organized in six main sections, and an appendix. The first three chapters
contain the background material, and therefore, can be skipped by specialists.

In Section \ref{sect-vanEst} we recall the classical van Est isomorphism to set up the background and the notation. We recall the (continuous) Lie group cohomology in Subsection \ref{subsect-cont-gr-cohom}, the Lie algebra cohomology in Subsection \ref{subsect-diff-forms-Lie-alg-cohom}, and the classical van Est isomorphism in Subsection \ref{subsect-vanEst}. 

Section \ref{sect-Hopf-cyclic-cohom}, contains the results and definitions we need from various Hopf-cyclic (co)homology theories with coefficients in both SAYD (stable anti-Yetter-Drinfeld) modules and SAYD contra-modules. To be more precise, following the brief survey of the Hopf-cyclic coefficient spaces, namely SAYD modules and SAYD contra-modules in Subsection \ref{subsect-Hopf-cyclic-coeff}, we collect the Hopf-cyclic (co)homologies (with SAYD module coefficients) of module (co)algebras in Subsection \ref{subsect-mod-alg-cohomol}, Subsection \ref{subsect-mod-coalg-cohomol}, and Subsection \ref{subsect-comod-alg-homol}. Then, we recall the Hopf-cyclic (co)homologies (with SAYD contra-module coefficents) for module (co)algebras in Subsection \ref{subsect-module-alg-contra-coeff}, Subsection \ref{subsect-module-coalg-contra-coeff}, and in Subsection \ref{subsect-contra-comod-alg}.

Section \ref{sect-cyclic-cohom-Lie-alg} consists of the technical results on (relative) cyclic (co)homology of Lie algebras that we shall need in the sequel. In Subsection \ref{subsect-unimodular-stable-AYD-over-Lie} we outline stable and unimodular stable AYD modules over Lie algebras, followed by the cyclic homology of a Lie algebra with coefficients in a stable AYD module in Subsection \ref{subsect-cyclic-hom-Lie}. Finally, in Subsection \ref{subsect-cyclic-cohom-Lie}, we recall the cyclic cohomology of a Lie algebra, with unimodular stable AYD module coefficients.

Section \ref{sect-vanEst-for-classical-Hopf} is where we prove the two van Est type isomorphisms for the classical Hopf algebras. We recall in Subsection \ref{subsect-coord-alg-funct} the Hopf-cyclic homology of the coordinate algebra of functions on an algebraic group with coefficients in an SAYD module. Here we also recalled the Hopf-cyclic cohomology of the coordinate Hopf algebra of an algebraic group, but this time with the coefficients in an SAYD contra-module.  Then, in Subsection \ref{subsect-vanEst-classical-Hopf-alg} we prove our first van Est isomorphisms between the Hopf-cyclic (co)homology of the coordinate algebra of functions on an algebraic group, and the relative Hopf-cyclic (co)homology of the universal enveloping algebra of its Lie algebra relative to the universal enveloping algebra of the Lie algebra of a maximal compact subgroup.  

The van Est isomorphisms for $h$-adic quantum groups and their corresponding $h$-adic enveloping algebras are proved in Section \ref{sect-h-adic-vanEst}. We first recall the $h$-adic Hopf algebra of functions for a Poisson-Lie group, and the corresponding $h$-adic enveloping algebra of its Lie (bi)algebra in Subsection \ref{subsect-quaintized-Hopf-alg}. Then in Subsection \ref{subsect-quantized-vanEst}, using the natural $h$-adic filtration on Hopf-cyclic complexes, we prove the van Est isomorphisms between the Hopf-cyclic (co)homology of the $h$-adic Hopf algebra of functions of a Poisson-Lie group, and the relative Hopf-cyclic (co)homology of the $h$-adic enveloping algebra of its Lie algebra relative to the quantized enveloping algebra of a maximal compact subalgebra.

The $q$-deformation analogues of the Hopf-cyclic van Est isomorphisms, on the other hand, are proved in Section \ref{sect-q-adic-vanEst}. We first recall the extended quantum enveloping algebras in Subsection \ref{subsect-Drinfeld-Jimbo}, and the corresponding coordinate algebras of functions on quantum (linear) groups in Subsection \ref{subsect-coord-alg-quant-gr}. Then, using the existence of nondegenerate pairings between these quantum objects, we prove in Subsection \ref{subsect-quant-vanEst-iso} isomorphisms between the Hopf-cyclic (co)homology of the coordinate algebra of a quantum linear group, and the relative Hopf-cyclic (co)homology of the extended quantum enveloping algebra of the corresponding Lie algebra relative to the quantum enveloping algebra of a maximal compact subalgebra.

Finally, in the Appendix we investigate how the Connes-Moscovici characteristic map and the
Hopf-cyclic analogues of the van Est isomorphisms are related via a single pairing, that we
named as the \emph{Janus map}, between Hopf-equivariant (co)cyclic objects.


\section{The van Est map}\label{sect-vanEst}

In the present section we shall recall the classical van Est isomorphism of \cite{vanEst53} from \cite[Sect. 5]{Dupo76}, \cite[Sect. 6]{HochMost62}, and \cite{Most61}. 

Following the conventions of \cite{Dupo76}, let $G$ be a connected semisimple real Lie group, $K \subseteq G$ a maximal compact subgroup, and $\G{g}$ be the Lie algebra of $G$. More generally, $G$ may be allowed to have finitely many connected components, see for instance \cite[Sect. 3]{Most61}.

\subsection{Continuous group cohomology}\label{subsect-cont-gr-cohom}~

Let $V$ be a \emph{continuous $G$-module}; that is, $V$ is a topological $\B{R}$-vector space equipped with a (left) $G$-module structure $G\times V \to V$ given by $(x,v)\mapsto x\cdot v$ which is continuous.

Now, $V$ being finite dimensional, let also $C_c^0(G,V)$ denotes the space of all continuous maps from $G$ to $V$, which is a continuous $G$-module via
\[
(x\rt c)(y) := x\cdot c(x^{-1}y),
\]
for any $x,y\in G$, and any $c\in C_c^0(G,V)$. Next, one defines inductively the  spaces of the higher cochains as 
\[
C^{n+1}_c(G,V):= C^0_c(G,C^n_c(G,V)), \qquad n\geq 0.
\]
Accordingly, for each $n\geq 0$, the space $C^n_c(G,V)$ may be identified with the space of all continuous maps from the product $G\times \ldots \times G$ of $(n+1)$-copies of $G$ to $V$, topologized by the compact-open topology, \cite[Sect. 2]{HochMost62}. Furthermore, each $C^n_c(G,V)$ is endowed with a continuous $G$-module structure given by
\[
(x\rt c)(y_0,\ldots,y_n) = x\cdot c(x^{-1}y_0,\ldots, x^{-1}y_n),
\] 
for any $x,y_0, \ldots, y_n \in G$. 

Finally, setting
\[
d:C^n_c(G,V) \to C^{n+1}_c(G,V), \qquad (d c)(x_0,\ldots,x_{n+1}) := \sum_{j=0}^{n+1}\,(-1)^j\,c(x_0,\ldots, \widehat{x}_j,\ldots, x_{n+1}),
\]
we arrive at an injective resolution
\begin{equation}\label{resolution-group}
\xymatrix{
0\ar[r] & V \ar[r] & C^0_c(G,V) \ar[r]^d & C^1_c(G,V) \ar[r]^{\,\,\,\,d} & \ldots
}
\end{equation}
of $V$, called the \emph{homogeneous resolution}, \cite[Sect. 2]{HochMost62}. Now, the homology of the $G$-invariant part of this (continuously) injective resolution\footnote[1]{As is shown in \cite[Sect. 2]{HochMost62} the homology of the $G$-fixed part is independent of the (continuously) injective resolution of the coefficient space.} is called the \emph{continuous group cohomology} of $G$ with coefficients in $V$, and it is denoted by $H^\ast_c(G,V)$.

As for the $G$-invariant part $C^n_c(G,V)^G$ of the space $C^n_c(G,V)$ of continuous $n$-cochains, it is worth mentioning that it may be identified with the space $C^{n-1}_c(G,V)$ of continuous $(n-1)$-cochains via
\[
\phi:C^n_c(G,V)^G\to C^{n-1}_c(G,V), \qquad \phi(c)(x_1,\ldots,x_n) := c(1,x_1,x_1x_2,\ldots,x_1\cdots x_n), \qquad n\geq 1
\]
whose inverse is given by
\[
\psi:C^{n-1}_c(G,V)\to C^n_c(G,V)^G, \qquad \psi(c)(x_0,\ldots,x_n) := x_0\cdot c(x_0^{-1}x_1,x_1^{-1}x_2,\ldots,x_{n-1}^{-1}x_n).
\]
In particular, $C^0_c(G,V)^G \cong V$ by $f\mapsto f(1)$.

As such, the continuous group cohomology may be computed by the \emph{non-homogeneous cochains} complex 
\[
\Big(C^\ast_c(G,V)^G,\,\d\Big) := \left(\bigoplus_{n\geq 0}\,C^n_c(G,V)^G,\, \d\right)
\]
where
\begin{align*}
& \d:C^n_c(G,V)^G \to C^{n+1}_c(G,V)^G, \qquad n\geq 1, \\
& \d c(x_1,\ldots,x_{n+1}) = \\
& \hspace{1.5cm} x_1\cdot c(x_2,\ldots,x_{n+1}) + \sum_{j=1}^n\,(-1)^{j+1}\,c(x_1,\ldots,x_jx_{j+1},\ldots x_{n+1}) + (-1)^{n+1}\, c(x_1,\ldots,x_n),
\end{align*}
and
\[
\d:C^0_c(G,V)^G \to C^1_c(G,V)^G, \qquad (\d v)(x) = x\cdot v - v. 
\]

\subsection{Differential forms and (relative) Lie algebra cohomology}\label{subsect-diff-forms-Lie-alg-cohom}~

Let, now, the $G$-action on $V$ be differentiable in the sense of \cite[Sect. 4]{HochMost62}, and $F_d(G,V)$ be the space of all differentiable (once again, in the sense of \cite[Sect. 4]{HochMost62}) maps from $G$ to $V$, which is topologized in such a way that a fundamental system of neighborhoods of 0 consists of the sets
\[
N(C,E,U) :=\{f\in F_d(G,V)\mid \d(f)(C)\subseteq U,\,\, \forall\,\d\in E\}
\]
where $C\subseteq G$ is a compact set, $E$ is a (finite) set of differential operators on $F_d(G,\B{R})$, and $U$ is a neighborhood of 0 in $V$. Accordingly, $F_d(G,V)$ is a continuous $G$-module with the $G$-action $G\times F_d(G,V) \to F_d(G,V)$ being $(x\cdot f) (y):= f(yx)$ for any $x,y\in G$, and any $f\in F_d(G,V)$.

Next, let $A^n(G,V)$ be the space of $V$-valued differential $n$-forms on $G$, which may be identified (regarding the Lie algebra elements as linear derivations on the algebra $F_d(G,\B{R})$ of differentiable functions on $G$) with $\wedge^n \G{g}^\vee \ot F_d(G,V)$, under which the $G$-action concentrates on $F_d(G,V)$, where $\G{g}^\vee:={\rm Hom}(\G{g},\B{R})$ refers to the linear dual, and
\begin{align}\label{Lie-alg-cohom-coboundary}
\begin{split}
& d:A^n(G,V)\to A^{n+1}(G,V), \\
&(d\a)(\xi_0,\ldots,\xi_n) :=\\
& \sum_{j=0}^n\,(-1)^j\,\xi_j\Big(\a(\xi_0,\ldots,\widehat{\xi}_j,\ldots, \xi_n)\Big) + \sum_{r<s}\,(-1)^{r+s}\,\a([\xi_r,\xi_s],\xi_0,\ldots,\widehat{\xi}_r,\ldots,\widehat{\xi}_s,\ldots,\xi_n)
\end{split}
\end{align}
for any $\a \in A^n(G,V)$, and any $\xi_0,\ldots,\xi_n\in \G{g}$. As such, we arrive at a differential complex 
\[
\xymatrix{
0\ar[r] & V\ar[r] &  A^0(G,V) \ar[r]^d & A^1(G,V) \ar[r]^{\,\,\,\,d} & \ldots
}
\]
where, $A^n(G,V)$ has the structure of a $G$-module given by
\[
(x\rt \a)(\xi_1,\ldots,\xi_n) := x\cdot \a(x^{-1}\rt\xi_1,\ldots,x^{-1}\rt\xi_n),
\]
for any $x\in G$, and any $\xi_1,\ldots,\xi_n \in \G{g}$, where
\[
(x\rt \xi)(f) := x\rt \xi(x^{-1}\rt f),
\]
for any $f\in F_d(G,\B{R})$, and 
\[
(x\rt f)(y) := f(x^{-1}y),
\]
for any $x,y\in G$. Then \eqref{Lie-alg-cohom-coboundary} respects the $G$-action, and the homology of the $G$-fixed part $A^n(G,V)^G\cong \wedge^n\G{g}^\vee \ot V$ captures the Lie algebra cohomology $H^\ast(\G{g},V)$ with coefficients in $V$.

Accordingly, the space $A^n(G/K,V)$ of $V$-valued differential $n$-forms on $G/K$ may be identified with $\wedge^n(\G{g}/\G{k})^\vee\ot F_d(G/K,V)$, where $\G{k}$ is the Lie algebra of $K$, and $F_d(G/K,V)\cong F_d(G,V)^K$. As such, the compatibility of \eqref{Lie-alg-cohom-coboundary} with the $G$-actions induces $d:A^n(G/K,V) \to A^{n+1}(G/K,V)$. 

Furthermore, as a result of the Poincar\'e Lemma (see also for instance \cite[Thm. 3.6.1]{Most61}, or \cite[Thm. 3.2]{Most55}),  
\begin{equation}\label{resolution-Lie-alg}
\xymatrix{
0\ar[r] & V \ar[r] & A^0(G/K,V) \ar[r]^d & A^1(G/K,V) \ar[r]^{\,\,\,\,\,\,\,\,\,\,\,\,d} & \ldots
}
\end{equation}
is a (continuously) injective resolution of $V$, \cite[Sect. 6]{HochMost62}. Finally, $A^n(G/K,V)^G\cong V\ot\wedge^n(\G{g}/\G{k})^\vee$, see  also \cite[Thm. 3.1]{HochKost62}, with the coboundary operator induced from \eqref{Lie-alg-cohom-coboundary}, reveals that the homology of the $G$-invariant part of the resolution \eqref{resolution-Lie-alg} coincides with the relative Lie algebra cohomology $H^\ast(\G{g},\G{k},V)$.

\subsection{The van Est isomorphism}\label{subsect-vanEst}~

Having two (continuously) injective resolutions \eqref{resolution-group} and \eqref{resolution-Lie-alg} of $V$, the homologies of their $G$-invariant parts are isomorphic; \cite[Thm. 6.1]{HochMost62}, \cite[3.6.1]{Most61}, and \cite{vanEst53}. The explicit isomorphism that identifies the cohomologies, on the other hand, is given in \cite[Thm. 5.1]{Dupo76}.

Along the lines of \cite{Dupo76}, we begin with $o:=\{K\}$, and continue for any $x_0\in G$ with the 0-simplex $\overline{\D}(x_0):=x_0^{-1}\cdot o$. Inductively (and relying on the fact that $G/K$ is diffeomorphic to an euclidean space), then, it is possible to set $\overline{\D}(x_0,\ldots,x_n)$ to be the geodesic cone of $\overline{\D}(x_1,\ldots,x_n)$ and $x_0^{-1}\cdot o$, the latter being the top point. 

Accordingly, there is a map
\[
\Phi:A^n(G/K,V) \to C_c^n(G,V), \qquad \Phi(\a)(x_0,\ldots,x_n):=\int_{\overline{\D}(x_0,\ldots,x_n)}\,\a,
\]
for any $\a \in A^n(G/K,V)$, and any $x_0,\ldots,x_n \in G$, which is a $G$-equivariant map that commutes with the respective differentials; \cite[Thm. 5.1]{Dupo76}. As such, it induces
\[
\Phi:A^n(G/K,V)^G \to C_c^n(G,V)^G,
\]
via which the isomorphism $H^\ast(\G{g},\G{k},V) \cong H^\ast_c(G,V)$ is achieved.

\section{Hopf-cyclic (co)homology with coefficients}\label{sect-Hopf-cyclic-cohom}

In this section we shall collect from \cite{Brze11,ConnMosc,HajaKhalRangSomm04-II,Hass14,HassKhalShap19,Rang11,KobyShap19} various Hopf-cyclic homology and Hopf-cyclic cohomology theories, with coefficients. Throughout the present section $k$ will stand for a field of characteristic zero, and the vector spaces will be assumed to be over it.

\subsection{Coefficient spaces for Hopf-cyclic cohomology}\label{subsect-Hopf-cyclic-coeff}~~

Let us first recall from \cite{HajaKhalRangSomm04-I} the notion of a stable anti-Yetter Drinfeld (SAYD in short) module.

Let $H$ be a $k$-Hopf algebra, and $M$ a right $H$-module and a left $H$-comodule, say by $M\ot H \to M$, $m \ot h\mapsto mh$ and $\nb:M\to H\ot M$, $m \mapsto m\ns{-1}\ot m\ns{0}$. Assume further that the two structures are compatible as 
\[
\nb(mh) = S(h\ps{3})m\ns{-1}h\ps{1} \ot m\ns{0}h\ps{2}, \qquad m\ns{0}m\ns{-1} = m
\]
for any $m\in M$, and any $h\in H$. $M$ is then said to be a \emph{right/left SAYD module} over $H$. See \cite[Def. 2.1]{HajaKhalRangSomm04-I} for the left/right, left/left and right/right versions.

We next recall the notion of a (right) contra-module over a coalgebra from \cite{Brze11}, see also \cite{BohmBrzeWisb09,Positselski-book,Rang11}.

A vector space $M$ along with a linear map $\a:\Hom(C,M)\to M$ that fits into the  (commutative) diagrams
\[
\xymatrix{
\Hom(C,\Hom(C,M))\ar[d]_\cong\ar[rrrr]^{\Hom(C,\a)} &&& &\Hom(C,M) \ar[d]^\a \\
 \Hom(C\ot C,M) \ar[rr]^{\Hom(\D,M)} & &\Hom(C,M) \ar[rr]^\a  & & M 
}
\]
where the vertical isomorphism on the left is the usual hom-tensor adjunction, and
\[
\xymatrix{
\Hom(k,M) \ar[rd]_\cong\ar[rr]^{\Hom(\ve,M)} & &  \Hom(C,M)\ar[dl]^\a \\
 &M  &
 }
\]
is called a \emph{(right) contra-module} over a coalgebra $C$.

Given a Hopf algebra $H$, a left $H$-module right $H$-contra-module $M$ is called a \emph{left/right SAYD contra-module} if 
\[ 
h\cdot\a(f) = \a(h\ps{2}\cdot f(S(h\ps{3})(\_)h\ps{1})), \qquad\a(r_\mu) = \mu
\]
for any $\mu\in M$, any $h\in H$, and any $f\in \Hom(H,M)$, where $r_\mu:H\to M$ is the mapping $h\mapsto h\cdot \mu$.

As is known, see for instance \cite{Brze11,Hass14,Rang11}, if $M$ is a left $C$-comodule by $\nb:M\to C\ot M$, then $M^\vee :=\Hom(M,k)$ is a right $C$-contra-module by
\[
\a:=\Hom(\nb,k):\Hom(C,M^\vee ) = \Hom(C,\Hom(M,k)) \cong \Hom(C\ot M,k) \to \Hom(M,k) = M^\vee ,
\]
more explicitly,
\[
\a(f)(m) = f(m\ns{-1} \ot m\ns{0}),
\]
for any $f\in \Hom(C,M^\vee )$, and any $m\in M$. Furthermore, if $M$ is a right/left SAYD module over $H$, then $M^\vee :=\Hom(M,k)$ is a left/right SAYD contra-module over $H$.

In the following three subsections we shall recall from \cite{HajaKhalRangSomm04-II} the Hopf-cyclic (co)homology, with SAYD module coefficients, associated to module algebra, module coalgebra, and comodule algebra symmetries.

\subsection{Hopf-cyclic cohomology of module algebras}\label{subsect-mod-alg-cohomol}~

Let $H$ be a Hopf algebra, and $A$ a (left) $H$-module algebra (say, by $\rt:H\ot A\to A$). Then,
\[
C_H(A,M) := \bigoplus_{n\geq0}\,C^n_H(A,M), \qquad C^n_H(A,M):=\Hom_H(M \ot A^{\ot\,(n+1)},k),
\]
where 
\[
(h\rt \vp)(m\ot \tilde{a}) := \vp(mh\ps{1} \ot S(h\ps{2})\cdot \tilde{a})
\]
for any $\tilde{a}:= a_0\odots a_n \in A^{\ot\,(n+1)}$, any $h \in H$, and any $m \in M$, may be endowed with a cocyclic structure by the cofaces
\begin{align*}
& d_i:C^{n-1}_H(A,M) \to C^n_H(A,M), \qquad 0\leq i \leq n, \\
& (d_i \vp)(m \ot a_0\odots a_n) := \vp(m\ot a_0 \odots a_ia_{i+1} \odots a_n), \qquad 0\leq i \leq n-1, \\
& (d_n \vp)(m \ot a_0\odots a_n) := \vp(m\ns{0} \ot (S^{-1}(m\ns{-1})\rt a_n)a_0 \ot a_1 \odots a_{n-1}),
\end{align*}
the codegeneracies
\begin{align*}
& s_j:C^{n+1}_H(A,M) \to C^n_H(A,M), \qquad 0\leq j \leq n, \\
& (s_j \vp)(m \ot a_0\odots a_n) := \vp(m\ot a_0 \odots a_j\ot 1 \ot a_{j+1} \odots a_n), 
\end{align*}
and the cyclic operator
\begin{align*}
& t_n:C^n_H(A,M) \to C^n_H(A,M),  \\
& (t_n \vp)(m \ot a_0\odots a_n) := \vp(m\ns{0} \ot S^{-1}(m\ns{-1})\rt a_n \ot a_0 \odots a_{n-1}).
\end{align*}
The cyclic (resp. periodic cyclic) homology of the above cocyclic module is denoted by $HC^\ast_H(A,M)$ (resp. $HP^\ast_H(A,M)$), and it is called the \emph{(periodic) Hopf-cyclic cohomology of the $H$-module algebra $A$, with coefficients in $M$}.

\subsection{Hopf-cyclic cohomology of module coalgebras}\label{subsect-mod-coalg-cohomol}~

Let $H$ be a Hopf algebra, and $C$ a left $H$-module coalgebra (say by $C\ot H\to H$, $c\ot h\mapsto c\cdot h$). Then, 
\[
C_H(C,M) := \bigoplus_{n\geq0}\,C^n_H(C,M), \qquad C^n_H(C,M):=M \ot_H C^{\ot\,(n+1)},
\]
where 
\[
h\rt\tilde{c} :=h\ps{1} \cdot c^0\ot h\ps{2}\cdot c^1 \odots h\ps{n+1}\cdot c^n
\]
for any $\tilde{c}:= c^0\odots c^n \in C^{\ot\,(n+1)}$, and any $h \in H$, may be endowed with a cocyclic structure by the cofaces
\begin{align*}
& d_i:C^{n-1}_H(C,M) \to C^n_H(C,M), \qquad 0\leq i \leq n, \\
& d_i (m \ot_H c^0\odots c^{n-1}) := m\ot_H c^0 \odots c^i\ps{1}\ot c^i\ps{2} \odots c^{n-1}, \qquad 0\leq i \leq n-1, \\
& d_n (m \ot_H c^0\odots c^{n-1}) := m\ns{0} \ot_H c^0\ps{2} \ot c^1 \odots c^{n-1} \ot m\ns{-1}\cdot c^0\ps{1},
\end{align*}
the codegeneracies
\begin{align*}
& s_j:C^{n+1}_H(C,M) \to C^n_H(C,M), \qquad 0\leq j \leq n, \\
& s_j (m \ot_H c^0\odots c^n) := m\ot_H c^0 \odots c^j\ot  \ve(c^{j+1}) \odots c^n, 
\end{align*}
and the cyclic operator
\begin{align*}
& t_n:C^n_H(C,M) \to C^n_H(C,M),  \\
& t_n(m \ot_H c^0\odots c^n) := m\ns{0} \ot_H c^1 \odots c^{n} \ot m\ns{-1}\cdot c^0.
\end{align*}
The cyclic (resp. periodic cyclic) homology of the above cocyclic module is denoted by $HC^\ast_H(C,M)$ (resp. $HP^\ast_H(C,M)$), and it is called the \emph{(periodic) Hopf-cyclic cohomology of the $H$-module coalgebra $C$, with coefficients in $M$}.

We next record the relative theory. Given a Hopf subalgebra $K\subseteq H$, let 
\begin{equation}\label{C-quotient-coalg}
\C{C}:=H\ot_K k \cong H/HK^+,
\end{equation}
where $\ve:H\to k$ being the counit of $H$, $K^+:=\ker\ve\vert_K$ is the augmentation ideal. The quotient coalgebra $\C{C}$ is a left $H$-module coalgebra by
\[
g \cdot \overline{h} := \overline{gh},
\]
and its Hopf-cyclic cohomology is called the relative Hopf-cyclic cohomology (with coefficients), \cite[Thm. 12]{ConnMosc}. More precisely, in this case
\[
C^n_H(\C{C},M) = M\ot_H \C{C}^{\ot\,n+1} \cong M \ot_K \C{C}^{\ot\,n} =:C^n(H,K,M)
\]   
via
\[
\Phi:C^n_H(\C{C},M) \to C^n(H,K,M), \qquad \Phi(m\ot_H c^0\odots c^n) := mh^0\ps{1}\ot_K S(h^0\ps{2})\cdot(c^1\odots c^n),
\]
considering $c^0 = \overline{h^0} \in \C{C}$, for $h^0\in H$. The cocyclic structure on 
\[
C(H,K,M) := \bigoplus_{n\geq 0}C^n(H,K,M)
\]
is given by the cofaces
\begin{align*}
& d_i:C^{n-1}_H(H,K,M) \to C^n_H(H,K,M), \qquad 0\leq i \leq n, \\
& d_0 (m \ot_K c^1\odots c^{n-1}) := m\ot_K \overline{1} \odots c^1 \odots c^{n-1}, \\
& d_i (m \ot_K c^1\odots c^{n-1}) := m\ot_K c^0 \odots c^i\ps{1}\ot c^i\ps{2} \odots c^{n-1}, \qquad 1\leq i \leq n-1, \\
& d_n (m \ot_K c^1\odots c^{n-1}) := m\ns{0} \ot_K c^1 \odots c^{n-1} \ot m\ns{-1},
\end{align*}
the codegeneracies
\begin{align*}
& s_j:C^{n+1}_H(H,K,M) \to C^n_H(H,K,M), \qquad 0\leq j \leq n, \\
& s_j (m \ot_K c^1\odots c^{n+1}) := m\ot_K c^1 \odots c^j\ot  \ve(c^{j+1}) \odots c^{n+1}, 
\end{align*}
and the cyclic operator
\begin{align*}
& t_n:C^n_H(H,K,M) \to C^n_H(H,K,M),  \\
& t_n(m \ot c^1\odots c^n) := m\ns{0}h^1\ps{1} \ot S(h^1\ps{2})\cdot (c^2 \odots c^{n-1} \ot m\ns{-1}),
\end{align*}
considering $\C{C}\ni c^1 = \overline{h^1}$ with $h^1\in H$. The cyclic (resp. periodic cyclic) homology of the above cocyclic module is denoted by $HC^\ast(H,K,M)$ (resp. $HP^\ast(H,K,M)$), and it is called the \emph{relative (periodic) Hopf-cyclic cohomology of $H$ relative to $K\subseteq H$, with coefficients in $M$}.

\subsection{Hopf-cyclic homology of comodule algebras}\label{subsect-comod-alg-homol}~

Let $H$ be a Hopf algebra, $A$ a right $H$-comodule algebra, say by  
\[
\Db:A\to A\ot H, \qquad \Db(a):=a^{\ps{0}}\ot  a^{\ps{1}},
\]
and let $M$ be a left/left SAYD module over $H$. There is, then, a cyclic structure on the complex
\[
C^H(A,M) := \bigoplus_{n\geq 0} C^H_n(A,M), \qquad C^H_n(A,M) := A^{\ot\,n+1}\,\square_H\,M
\]
given by the faces
\begin{align*}
& \d_i:C^H_n(A,M) \to C^H_{n-1}(A,M), \qquad 0\leq i \leq n, \\
& \d_i (a_0\odots a_n \ot m) := a_0 \odots a_ia_{i+1} \odots a_n\ot m, \qquad 0\leq i \leq n-1, \\
& \d_n(a_0\odots a_n\ot m) := a_n^{\ps{0}}a_0 \ot a_1 \odots a_{n-1}\ot  a_n^{\ps{1}}m,
\end{align*}
the codegeneracies
\begin{align*}
& \s_j:C^H_n(A,M) \to C^H_{n+1}(A,M), \qquad 0\leq j \leq n, \\
& \s_j(a_0\odots a_n\ot m) := a_0 \odots a_j\ot 1 \ot a_{j+1} \odots a_n\ot m, 
\end{align*}
and the cyclic operator
\begin{align*}
& \tau_n:C^H_n(A,M) \to C^H_n(A,M),  \\
& \tau_n (a_0\odots a_n\ot m) := a_n^{\ps{0}}\ot a_0 \ot a_1 \odots a_{n-1} \ot a_n^{\ps{1}}m.
\end{align*}
The cyclic (resp. periodic cyclic) homology of the above cocyclic module is denoted by $HC^H_\ast(A,M)$ (resp. $HP^H_\ast(A,M)$), and it is called the \emph{(periodic) Hopf-cyclic homology of the $H$-comodule algebra $A$, with coefficients in $M$}.

On the last three subsections, on the other hand, we shall record from \cite{Brze11,Rang11} the Hopf-cyclic (co)homologies, with contra-module coefficients, of module algebras, module coalgebras, and comodule algebras.

\subsection{Hopf-cyclic cohomology (with contramodule coefficients) of module algebras}\label{subsect-module-alg-contra-coeff}~

Given a left $H$-module algebra $A$, and a right/left SAYD module $M$ over $H$, let $C^n_H(A,M^\vee ) := \Hom_H(A^{\ot\,(n+1)},M^\vee )$ be the space of left $H$-linear maps. Then, it follows from \cite[Prop. 2.2]{Rang11} that the isomorphisms
\[
\C{I}:C^n_H(A,M)\to C^n_H(A,M^\vee ), \qquad \C{I}(\vp)(a_0\odots a_n)(m) := \vp(m\ot a_0 \odots a_n)
\]
and
\[
\C{J}:C^n_H(A,M^\vee )\to C^n_H(A,M), \qquad \C{J}(\phi)(m\ot a_0\odots a_n) := \phi(a_0 \odots a_n)(m)
\]
pull the cocyclic structure on $C_H(A,M)$ onto 
\[
C_H(A,M^\vee ) := \bigoplus_{n\geq0}\,C^n_H(A,M^\vee ), \qquad C^n_H(A,M^\vee ):=\Hom_H(A^{\ot\,(n+1)},M^\vee ).
\]
The resulting cocyclic structure is given explicitly by the cofaces
\begin{align*}
& \G{d}_i:C^{n-1}_H(A,M^\vee ) \to C^n_H(A,M^\vee ), \qquad 0\leq i \leq n, \\
& (\G{d}_i \phi)(a_0\odots a_n) := \phi(a_0 \odots a_ia_{i+1} \odots a_n), \qquad 0\leq i \leq n-1, \\
& (\G{d}_n \phi)(a_0\odots a_n) := \a(\phi( (S^{-1}(\_)\rt a_n)a_0 \ot a_1 \odots a_{n-1})),
\end{align*}
the codegeneracies
\begin{align*}
& \G{s}_j:C^{n+1}_H(A,M^\vee ) \to C^n_H(A,M^\vee ), \qquad 0\leq j \leq n, \\
& (\G{s}_j \phi)(a_0\odots a_n) := \phi( a_0 \odots a_j\ot 1 \ot a_{j+1} \odots a_n), 
\end{align*}
and the cyclic operator
\begin{align*}
& \G{t}_n:C^n_H(A,M^\vee ) \to C^n_H(A,M^\vee ),  \\
& (\G{t}_n \phi)(a_0\odots a_n) := \a(\phi(S^{-1}(\_)\rt a_n \ot a_0 \odots a_{n-1})).
\end{align*}
The cyclic (resp. periodic cyclic) homology of this cocyclic module is denoted by $HC^\ast_H(A,M^\vee )$ (resp. $HP^\ast_H(A,M^\vee )$), and it is called the \emph{(periodic) Hopf-cyclic cohomology of the $H$-module algebra $A$, with coefficients in $M^\vee $}.

\begin{remark}\label{remark-mod-alg-contra-coeff}
Let $H$ be a Hopf algebra, $A$ a left $H$-module algebra, and $M$ a right/left SAYD module over $H$. In view of the hom-tensor adjunction
\[
\Hom_H(A^{\ot\,n+1},M^\vee ) \cong \Hom(M\ot_H A^{\ot\,n+1}, k),
\]
and hence the pairing
\[
\langle \,,\rangle: C^n_H(A,M^\vee )\ot C_{n,H}(A,M) \to k, \qquad \langle\vp,m\ot_H a_0\odots a_n\rangle := \vp(a_0\odots a_n)(m),
\] 
the Hopf-cyclic cohomology with SAYD contra-module coefficients of a module algebra is obtained by dualizing the Hopf-cyclic homology of the same (module-)algebra with SAYD coeffcients\footnote{The Hopf-cyclic homology of the (left) $H$-module algebra $A$, with coeffcients in the SAYD module $M$ over $H$ is computed by the complex
\[
C_H(A,M):= \bigoplus_{n\geq 0}C_{n,H}(A,M), \qquad C_{n,H}(A,M):=M\ot_H A^{\ot\,n+1}.
\]}. 
\end{remark}

\subsection{Hopf-cyclic homology (with contramodule coefficients)  of module coalgebras}\label{subsect-module-coalg-contra-coeff}~

Given a left $H$-module coalgebra $C$, and a right/left SAYD module $M$ over $H$, let $C_{n,H}(C,M^\vee ) := \Hom_H(C^{\ot\,(n+1)}, M^\vee)$ be the space of left $H$-linear maps from $C^{\ot\,(n+1)}$ to $M^\vee$. Then, it follows from \cite[Sect. 4]{Brze11} that the cocyclic structure on $C_H(C,M)$ induces a cyclic structure on the complex
\[
C_H(C,M^\vee ) := \bigoplus_{n\geq0}\,C_{n,H}(C,M^\vee )
\]
via the pairing
\begin{equation}\label{pairing-coalgebra-complexes}
\langle\, ,\,\rangle:C_{n,H}(C,M^\vee )\ot C_H^n(C,M) \to k, \qquad \langle \psi,\,m\ot_H c^0\odots c^n\rangle:=\psi(c^0 \odots c^n)(m).
\end{equation}
The resulting cyclic structure, then, may be given by the faces
\begin{align*}
& \d_i:C_{n,H}(C,M^\vee ) \to C_{n-1,H}(C,M^\vee ), \qquad 0\leq i \leq n, \\
& (\d_i \psi)(c^0\odots c^{n-1}) := \psi(c^0 \odots c^i\ps{1}\ot c^i\ps{2} \odots c^{n-1}), \qquad 0\leq i \leq n-1, \\
& (\d_n \psi)(c^0\odots c^{n-1}) := \a(\psi( c^0\ps{2}\ot c^1 \odots c^{n-1} \ot (\_)\cdot c^0\ps{1})),
\end{align*}
the degeneracies
\begin{align*}
& \s_j:C_{n,H}(C,M^\vee ) \to C_{n+1,H}(C,M^\vee ), \qquad 0\leq j \leq n, \\
& (\s_j \psi)(c^0\odots c^{n+1}) := \ve(c^{j+1})\psi(c^0 \odots c^j\ot c^{j+2} \odots c^{n+1}), 
\end{align*}
and the cyclic operator
\begin{align*}
& \tau_n:C_{n,H}(C,M^\vee ) \to C_{n,H}(C,M^\vee ),  \\
& (\tau_n \psi)(c^0\odots c^n) := \a(\psi(c^1 \odots c^n \ot (\_)\cdot c^0)).
\end{align*}
The cyclic (resp. periodic cyclic) homology of this cyclic module is denoted by $HC_{\ast,H}(C,M^\vee )$ (resp. $HP_{\ast,H}(C,M^\vee )$), and it is called the \emph{(periodic) Hopf-cyclic homology of the $H$-module coalgebra $C$, with coefficients in $M^\vee $}.

Now, given a Hopf-subalgebra $K\subseteq H$, let $\C{C}$ be the left $H$-module coalgebra of \eqref{C-quotient-coalg}. It then follows at once that 
\[
C_{n,H}(\C{C},M^\vee ) \cong C_n(H,K,M^\vee ) := \Hom_K(\C{C}^{\ot\,n},M^\vee )
\]
via 
\begin{align*}
& \Lambda: C_{n,H}(\C{C},M^\vee ) \to C_n(H,K,M^\vee ),\\
& \Lambda(\psi)(c^1\odots c^n) := \psi(\overline{1}\ot c^1\odots c^n),
\end{align*}
whose inverse is given by
\begin{align*}
& \Lambda^{-1}: C_n(H,K,M^\vee ) \to C_{n,H}(\C{C},M^\vee ),\\
& \Lambda^{-1}(\phi)(c^0\odots c^n) := h^0\ps{1}\rt \phi(S(h^0\ps{2})\cdot (c^1\odots c^n)),
\end{align*}
for any $\phi\in C_n(H,K,M^\vee )$, where we consider $c^i := \overline{h^i} \in \C{C}$ for $0\leq i \leq n$. Accordingly, the cyclic structure on $C_H(\C{C},M^\vee )$ gives rise to a cyclic structure on 
\[
C(H,K,M^\vee ) := \bigoplus_{n\geq0}\,C_n(H,K,M^\vee )
\]
consisting of the faces
\begin{align*}
& \d_i:C_n(H,K,M^\vee ) \to C_{n-1}(H,K,M^\vee ), \qquad 0\leq i \leq n, \\
& (\d_0 \phi)(c^1\odots c^{n-1}) :=\phi(1\ot c^1 \odots c^{n-1}), \\
& (\d_i \phi)(c^1\odots c^{n-1}) := \phi(c^1 \odots c^i\ps{1}\ot c^i\ps{2} \odots c^{n-1}), \qquad 0\leq i \leq n-1, \\
& (\d_n \phi)(c^1\odots c^{n-1}) := \a(\phi( c^1 \odots c^{n-1} \ot (\_))),
\end{align*}
the degeneracies
\begin{align*}
& \s_j:C_n(H,K,M^\vee ) \to C_{n+1}(H,K,M^\vee ), \qquad 0\leq j \leq n, \\
& (\s_j \phi)(c^1\odots c^{n+1}) := \ve(c^{j+1})\phi(c^1 \odots c^j\ot c^{j+2} \odots c^{n+1}), 
\end{align*}
and the cyclic operator\footnote[1]{Although the presentation of the cyclic operator seems different from the one in \cite[Def. 13]{ConnMosc}, they are the same when evaluated on an $m\in M$. Indeed, 
\begin{align*}
& (\tau_n \phi)(c^1\odots c^n)(m) =  \a(\phi(S(h^1\ps{3})\cdot(c^2 \odots c^n) \ot (\_)S^{-1}(h^1\ps{1})))(mh^1\ps{2}) = \\
& \a(\phi(S(h^1\ps{3})\cdot(c^2 \odots c^n) \ot (mh^1\ps{2})\ns{-1}S^{-1}(h^1\ps{1}))((mh^1\ps{2})\ns{0}) = \\
& \phi(S(h^1\ps{3})\cdot(c^2 \odots c^n) \ot S(h^1\ps{2}\ps{3})m\ns{-1}h^1\ps{2}\ps{1}S^{-1}(h^1\ps{1}))(mh^1\ps{2}\ps{2}) = \\
& \phi(S(h^1\ps{2})\cdot(c^2 \odots c^n\ot m\ns{-1}))(mh^1\ps{1}).
\end{align*}}
\begin{align*}
& \tau_n:C_n(H,K,M^\vee ) \to C_n(H,K,M^\vee ),  \\
& (\tau_n \phi)(c^1\odots c^n) := h^1\ps{2}\rt \a(\phi(S(h^1\ps{3})\cdot(c^2 \odots c^n) \ot (\_)S^{-1}(h^1\ps{1}))),
\end{align*}
see also \cite[Def. 13]{ConnMosc}. The cyclic (resp. periodic cyclic) homology of this cyclic module is denoted by $HC_\ast(H,K,M^\vee )$ (resp. $HP_\ast(H,K,M^\vee )$), and it is called the \emph{relative (periodic) Hopf-cyclic homology of the Hopf algebra $H$, relative to $K \subseteq H$, with coefficients in $M^\vee $}.

\begin{remark}
As we have noted in Remark \ref{remark-mod-alg-contra-coeff}, let us record here that the Hopf-cyclic homology of a module coalgebra, with SAYD contra-module coefficients, is obtained by dualizing the Hopf-cyclic complex computing the Hopf-cyclic cohomology of the same module coalgebra, with SAYD module coefficients.
\end{remark}

\subsection{Hopf-cyclic cohomology (with contramodule coefficients)  of comodule algebras}~\label{subsect-contra-comod-alg}

We shall now apply the strategy of Subsection \ref{subsect-module-alg-contra-coeff} and Subsection \ref{subsect-module-coalg-contra-coeff} to obtain the Hopf-cyclic cohomology, with contra-module coefficients, of comodule algebras. More precisely, we shall dualise the complex computing the Hopf-cyclic homology (with coefficients in a SAYD module) of a comodule algebra, to obtain the Hopf-cyclic cohomology (with coefficients in a SAYD contra-module) of a comodule algebra.

Let $H$ be a Hopf algebra, $A$ a right $H$-comodule algebra (with the notation used above), and let $M$ be a left/left SAYD module over $H$. Let also $M^\vee :=\Hom_k(M,k)$, which is a right/right SAYD contra-module over $H$. We shall, accordingly, consider the complex
\begin{align*}
& C^H(A,M^\vee ) := \bigoplus_{n\geq 0} C^{n,H}(A,M^\vee ), \\ 
& C^{n,H}(A,M^\vee ) := \Hom(A^{\ot\,n+1}\,\square_H\,M,k) \cong \Hom(A^{\ot\,n+1}, k)\ot_{H^\circ}M^\vee .
\end{align*}
The duality
\[
\langle\,,\rangle: C^{n,H}(A,M^\vee )\ot C_n^H(A,M)\to k,\qquad \langle \phi\ot_{H^\circ}f,a_0\odots a_n\ot m\rangle := \phi(a_0\odots a_n)f(m)
\]
then, uses the cyclic structure on $C^H(A,M)$ in order to induce a cocyclic structure on $C^H(A,M^\vee )$ through the cofaces
\begin{align*}
& \G{d}_i:C^{n-1,H}(A,M^\vee ) \to C^{n,H}(A,M^\vee ), \qquad 0\leq i \leq n, \\
& (\G{d}_i (\phi\ot_{H^\circ}f))(a_0\odots a_n) := \phi(a_0 \odots a_ia_{i+1} \odots a_n) f, \qquad 0\leq i \leq n-1, \\
& (\G{d}_n (\phi\ot_{H^\circ}f))(a_0\odots a_n) := \phi(a_n^{\ps{0}}a_0 \ot a_1 \odots a_{n-1})(f\lt a_n^{\ps{1}}),
\end{align*}
the codegeneracies
\begin{align*}
& \G{s}_j:C^{n+1,H}(A,M^\vee ) \to C^{n,H}(A,M^\vee ), \qquad 0\leq j \leq n, \\
& (\G{s}_j (\phi\ot_{H^\circ}f))(a_0\odots a_n) := \phi( a_0 \odots a_j\ot 1 \ot a_{j+1} \odots a_n) f, 
\end{align*}
and the cyclic operator
\begin{align*}
& \G{t}_n:C^{n,H}(A,M^\vee ) \to C^{n,H}(A,M^\vee ),  \\
& (\G{t}_n (\phi\ot_{H^\circ}f))(a_0\odots a_n) := \phi(a_n^{\ps{0}}\ot a_0 \ot a_1 \odots a_{n-1})(f\lt a_n^{\ps{1}}).
\end{align*}
The cyclic (resp. periodic cyclic) homology of this cocyclic module is denoted by $HC^{\ast,H}(A,M^\vee )$ (resp. $HP^{\ast,H}(A,M^\vee )$), and it is called the \emph{(periodic) Hopf-cyclic cohomology of the $H$-comodule algebra $A$, with coefficients in $M^\vee $}.

\section{Cyclic (co)homologies of Lie algebras}\label{sect-cyclic-cohom-Lie-alg}

The main results of the present manuscript, that is the quantum van Est isomorphisms between (relative) Hopf-cyclic (co)homologies of quantized enveloping algebras and Hopf-cyclic (co)homologies of quantized function algebras, follow (in the $E_1$-levels of relevant spectral sequences) from isomorphisms between (relative) cyclic (co)homologies of Lie algebras and (by a slight abuse of language) cyclic (co)homologies of Lie groups.

We shall, accordingly, recall now the (relative) cyclic homology and cyclic cohomology of Lie algebras, with coefficients in (unimodular) SAYD modules. We, on the other hand, continue to adopt the convention to work over a ground field $k$ of characteristic 0.

\subsection{Cyclic (co)homological coefficient spaces over Lie algebras}\label{subsect-unimodular-stable-AYD-over-Lie}~

Along the way to cyclic (co)homology of Lie algebras, we shall first recall the appropriate coefficient spaces.

Following \cite{RangSutl-II} let $\G{g}$ be a Lie algebra, and let $M$ be a \emph{right/left SAYD module over $\G{g}$}, that is,
\begin{itemize}
\item[(i)] $M$ is a \emph{right $\G{g}$-module}, in other words,
\[
m[X_1,X_2] = (mX_1)X_2 - (mX_2)X_1
\]
for any $m\in M$, and any $X_1,X_2 \in \G{g}$,
\item[(ii)] $M$ is a \emph{left $\G{g}$-comodule}, or equivalently, there is $\nb:M\to \G{g}\ot M$, $\nb(m):= m\nsb{-1}\ot m\nsb{0}$, so that
\[
m\nsb{-2}\wedge m\nsb{-1} \ot m\nsb{0} = 0,
\]
for any $m\in M$, where $m\nsb{-2}\ot m\nsb{-1} \ot m\nsb{0} := m\nsb{-1}\ot m\nsb{0}\nsb{-1}\ot m\nsb{0}\nsb{0}$,
\item[(iii)] $M$ is a \emph{right/left AYD module over $\G{g}$}, in other words,
\[
\nb(mX) = m\nsb{-1}\ot m\nsb{0}X + [m\nsb{-1},X]\ot m\nsb{0}
\]
for any $m\in M$, and any $X\in \G{g}$, and finally
\item[(iv)] $M$ is \emph{stable}, that is,
\[
m\nsb{0}m\nsb{-1} = 0,
\]
for any $m\in M$.
\end{itemize}

Let us note also that $M$ is stable if and only if
\[
(m\t^i)\xi_i = 0,
\]
where $\{\xi_i\mid 1\leq i\leq \dim(\G{g})\}$, $\{\t^i\mid 1\leq i\leq \dim(\G{g})\}$ is a dual pair of basis, for $\G{g}$ and $\G{g}^\vee$ respectively, and we consider the right $S(\G{g}^\vee)$-action as
\[
m\t := \t(m\nsb{-1})m\nsb{0}
\]
for any $m\in M$, and any $\t\in S(\G{g}^\vee)$. In this language, a right $\G{g}$-module left $\G{g}$-comodule $M$ is said to be \emph{unimodular stable} if
\[
(m\xi_i)\t^i = 0.
\]

\subsection{Cyclic homology of Lie algebras}\label{subsect-cyclic-hom-Lie}~

In the present subsection we shall recall, from \cite{RangSutl-II}, the cyclic homology of a Lie algebra, and its relation with the Hopf-cyclic cohomology of the universal enveloping algebra of this Lie algebra.

To begin with, let $M$ be a right/left stable AYD module over a Lie algebra $\G{g}$. Then, 
\[
C_{r,s}(\G{g},M):=\begin{cases}
M\ot \wedge^{s-r}\G{g} & \text{ if } 0\leq r\leq s, \\
0 & \text{ otherwise},
\end{cases}
\]
is a bicomplex with the differentials 
\begin{align*}
& \p_{CE}:C_{r,s}(\G{g},M) \to C_{r+1,s}(\G{g},M), \\
& \p(m\ot X_1\wedge \ldots \wedge X_n) := \\
& \sum_{1\leq i<j\leq n}\,(-1)^{i+j-1}\,m\ot [X_i,X_j]\wedge X_1\wedge\ldots \wedge \widehat{X}_i \wedge\ldots \wedge \widehat{X}_j \wedge\ldots \wedge X_n + \\
&\hspace{2cm} \sum_{1\leq i\leq n}\,(-1)^{i+1}\, mX_i \ot X_1 \wedge\ldots \wedge \widehat{X}_i \wedge\ldots \wedge X_n,
\end{align*}
and
\begin{align*}
& \p_K:C_{r,s}(\G{g},M) \to C_{r,s+1}(\G{g},M), \\
& \p_K(m\ot X_1 \wedge\ldots \wedge X_n) := m\t^i\ot \xi_i\wedge X_1 \wedge\ldots \wedge X_n.
\end{align*}
The (total) homology of this bicomplex is denoted by $HC^\ast(\G{g},M)$, and it is called the \emph{cyclic homology of $\G{g}$, with coefficients in $M$}.

Similarly, the (total) homology of the bicomplex
\[
C_{r,s}(\G{g},M):=\begin{cases}
M\ot \wedge^{s-r}\G{g} & \text{ if }  r\leq s, \\
0 & \text{ otherwise},
\end{cases}
\]
is denoted by $HP^\ast(\G{g},M)$, and it is called the \emph{periodic cyclic homology of $\G{g}$, with coefficients in $M$}.

Two more remarks are in order.

If an SAYD module $M$ over $\G{g}$ is \emph{locally conilpotent}, that is, for any $m\in M$ there is $p\in \B{N}$ so that $\nb^p(m) = 0$, then $M$ exponentiates to an SAYD module over $U(\G{g})$,  \cite[Prop. 5.10]{RangSutl-II}. Furthermore, in this case,
\[
HC^\ast(U(\G{g}),M) \cong HC^\ast(\G{g},M).
\]
Let us note also that the cyclic homology theory for Lie algebras may be relativized. More precisely, given a Lie subalgebra $\G{h}\subseteq \G{g}$, the relative cyclic (resp. periodic cyclic) homology $HC^\ast(\G{g},\G{h},M)$ (resp. $HP^\ast(\G{g},\G{h},M)$) is defined to be the homology of the bicomplex
\[
C_{r,s}(\G{g},\G{h},M):=\begin{cases}
M\ot \wedge^{s-r}(\G{g}/\G{h}) & \text{ if } 0\leq r\leq s, \\
0 & \text{ otherwise},
\end{cases}
\]
(resp.
\[
C_{r,s}(\G{g},\G{h},M):=\begin{cases}
M\ot \wedge^{s-r}(\G{g}/\G{h}) & \text{ if } r\leq s, \\
0 & \text{ otherwise}.)
\end{cases}
\]
If, in addition, $M$ is locally conilpotent, then \cite[Thm. 6.2]{RangSutl-II} yields at once that
\begin{equation}\label{Lie-alg-cyclic-hom}
HC^\ast(U(\G{g}),U(\G{h}),M) \cong HC^\ast(\G{g},\G{h},M).
\end{equation}

\subsection{Cyclic cohomology of Lie algebras}\label{subsect-cyclic-cohom-Lie}~

We shall now develop the Lie algebra cyclic cohomology analogue of \eqref{Lie-alg-cyclic-hom} above, which was not treated earlier in \cite{RangSutl-II}.

Let now $V$ be a right/left unimodular stable AYD module over $\G{g}$. Then, along the lines of \cite{RangSutl-II},
\[
W^{r,s}(\G{g},V):=\begin{cases}
V\ot \wedge^{s-r}\G{g}^\vee & \text{ if } 0\leq r\leq s, \\
0 & \text{ otherwise},
\end{cases}
\]
is a bicomplex with the differentials 
\begin{align*}
& d_{CE}:W^{r,s}(\G{g},V) \to W^{r-1,s}(\G{g},V), \\
& d_{CE}\vp(X_1\wedge \ldots \wedge X_{n+1}) := \\
& \sum_{1\leq i<j\leq n}\,(-1)^{i+j-1}\,\vp([X_i,X_j]\wedge X_1\wedge\ldots \wedge \widehat{X}_i \wedge\ldots \wedge \widehat{X}_j \wedge\ldots \wedge X_{n+1}) + \\
&\hspace{2cm} \sum_{1\leq i\leq n}\,(-1)^{i+1}\, \vp(X_1 \wedge\ldots \wedge \widehat{X}_i \wedge\ldots \wedge X_{n+1}) X_i,
\end{align*}
and
\begin{align*}
& d_K:W^{r,s}(\G{g},V) \to W^{r,s-1}(\G{g},V), \\
& d_K\vp(X_1 \wedge\ldots \wedge X_{n-1}) := \vp(\xi_i\wedge X_1 \wedge\ldots \wedge X_{n-1})\t^i.
\end{align*}
The (total) homology of this bicomplex is denoted by $HC_\ast(\G{g},V)$, and it is called the \emph{cyclic cohomology of $\G{g}$, with coefficients in $V$}.

Similarly, the (total) homology of the bicomplex
\[
W^{m,n}(\G{g},V):=\begin{cases}
V\ot \wedge^{n-m}\G{g}^\vee & \text{ if }  m\leq n, \\
0 & \text{ otherwise},
\end{cases}
\]
is denoted by $HP_\ast(\G{g},V)$, and is called the \emph{periodic cyclic cohomology of $\G{g}$, with coefficients in $V$}.

Just as the cyclic homology of Lie algebras, cyclic cohomology theory for Lie algebras may be relativized as well. Given a Lie subalgebra $\G{h}\subseteq \G{g}$, the relative cyclic (resp. periodic cyclic) cohomology $HC_\ast(\G{g},\G{h},V)$ (resp. $HP_\ast(\G{g},\G{h},V)$) is defined to be the homology of the bicomplex
\[
W^{m,n}(\G{g},\G{h},V):=\begin{cases}
V\ot \wedge^{n-m}(\G{g}/\G{h})^\vee & \text{ if } 0\leq m\leq n, \\
0 & \text{ otherwise},
\end{cases}
\]
(resp.
\[
W^{m,n}(\G{g},\G{h},V):=\begin{cases}
V\ot \wedge^{n-m}(\G{g}/\G{h})^\vee & \text{ if } m\leq n, \\
0 & \text{ otherwise}.)
\end{cases}
\]
Along the way to prove an analogue of \eqref{Lie-alg-cyclic-hom} for Lie algebra cyclic cohomology, we shall first record the following auxiliary result on the coefficient spaces.

\begin{proposition}\label{prop-stable-to-unimodular-stable}
If $M$ is a right/left stable AYD module over $\G{g}$, then $M^\vee$ is a right/left unimodular stable AYD module over $\G{g}$.
\end{proposition}

\begin{proof}
We recall from \cite[Prop. 5.13]{RangSutl-II} that $M$ is a right/left stable AYD module over $\G{g}$ if and only if it is a stable right module over the (semi-direct sum) Lie algebra $\G{g}^\vee\rtimes \G{g}$, that is,
\[
(m \xi)\t - (m\t)\xi = m(\xi\rt \t)
\]
for any $m\in M$, any $\xi \in \G{g}$, and any $\t \in \G{g}^\vee$. Accordingly, $M$ happens to be a right module over the semi-direct product algebra $U(\G{g}^\vee\rtimes \G{g}) \cong S(\G{g}^\vee)\rtimes U(\G{g})$.

On the other hand, $M$ being a right module over $\G{g}$, and over $S(\G{g}^\vee)$, $M^\vee $ bears a natural right $\G{g}$-module structure
\[
(f\lt \xi)(m) := -f(m\xi),
\]
and a natural right $S(\G{g}^\vee)$-module structure by
\[
(f\lt \t)(m) := f(m\t)
\]
for any $f\in M^\vee $, any $m\in M$, any $\xi\in \G{g}$, and any $\t\in\G{g}^\vee$. Then,
\begin{align*}
& ((f \lt\xi)\lt\t - (f\lt \t)\lt\xi)(m) = f(-(m\t)\xi + (m\xi)\t ) = \\
& f(m(\xi\rt \t)) = (f \lt (\xi\rt \t))(m)
\end{align*}
that is, $M^\vee $ is a right $\G{g}^\vee\rtimes \G{g}$-module. Furthermore, $M^\vee $ is unimodular stable. Indeed, for a dual pair of bases $\{\xi_i\mid 1\leq i\leq \dim(\G{g})\}$ and $\{\t^i\mid 1\leq i\leq \dim(\G{g})\}$, 
\[
((f\lt \xi_i)\lt \t^i)(m) = -f((m\t^i)\xi_i) = 0,
\]
for any $f\in M^\vee $, any $m\in M$, any $\xi\in \G{g}$, and any $\t\in\G{g}^\vee$.
\end{proof}

The following is the main result of this subsection.

\begin{proposition}\label{prop-U(g)-contra-module-coeff}
Let $M$ be a locally conilpotent stable AYD module over a Lie algebra $\G{g}$, let $M^\vee$ be the corresponding AYD contramodule, and let also $\G{h}\subseteq \G{g}$ be a Lie subalgebra. Then
\[
HC_\ast(U(\G{g}),U(\G{h}),M^\vee ) \cong HC_\ast(\G{g},\G{h},M^\vee ).
\]
\end{proposition}

\begin{proof}
Let us first note that if $M$ is a locally conilpotent right/left stable AYD over $\G{g}$, then it follows from \cite[Prop. 5.10\,\&\, Lemma 5.11]{RangSutl-II} that it is a right/left SAYD module over $U(\G{g})$, and hence $M^\vee $ is a left/right SAYD contra-module over $U(\G{g})$. Thus, the homology on the left hand side is well-defined. The homology on the right hand side, on the other hand, is defined in view of Proposition \ref{prop-stable-to-unimodular-stable}.

Now, being an AYD module over $U(\G{g})$, $M$ admits an increasing filtration $(F_pM)_{p\in \B{Z}}$, so that $F_pM / F_{p-1}M$ is a trivial $U(\G{g})$-comodule for any $p\in \B{Z}$, \cite[Lemma 6.2]{JaraStef06}. Accordingly, an isomorphism on the level of the $E_1$-terms of the associated spectral sequences is given by \cite[Thm. 15]{ConnMosc}. More precisely, the Hochschild boundary map of the left hand side coincides with the Koszul boundary map on the right hand side, and the Connes coboundary map of the former coincides with the Chevalley-Eilenberg coboundary map of the latter.
\end{proof}

The pairing \eqref{pairing-coalgebra-complexes} between the complexes computing $HC_\ast(U(\G{g}),U(\G{h}),M^\vee )$ and $HC^\ast(U(\G{g}),U(\G{h}),M)$, yields at once the following analogue of \cite[Thm. 2]{BlanWign83}. 

\begin{corollary}\label{BlancWigner-Lie-alg-cyclic-version}
Given two Lie algebras $\G{h}\subseteq \G{g}$, and a SAYD module $M$,
\[
HC_\ast(\G{g},\G{h},M^\vee ) = HC^\ast(\G{g},\G{h},M)^\vee.
\]
\end{corollary}

\section{van Est isomorphism on classical Hopf algebras}\label{sect-vanEst-for-classical-Hopf}

In the present section we shall present two van Est type isomorphisms (one on the level of Hopf-cyclic homology, and one on the level of Hopf-cyclic cohomology) for the classical Hopf algebras, by which we mean the universal enveloping algebra of a Lie algebra, and the Hopf algebra of functions on a Lie or algebraic group. 

\subsection{Hopf-cyclic complexes for the Hopf algebras of functions}\label{subsect-coord-alg-funct}~

Having established the (relative) Hopf-cyclic cohomology for the universal enveloping algebra $U(\G{g})$ of a Lie algebra $\G{g}$ through Subsection \ref{subsect-mod-coalg-cohomol}, this Hopf-cyclic complex is now in need of a companion to be able to talk about an isomorphism between the two. Below we shall obtain the relevant complex by dualizing the (relative) Hopf-cyclic complex of $U(\G{g})$.

As such, the integral part of the construction is to find a dual object for $U(\G{g})$, which (inspired by the classical van Est isomorphism) is ought to be related to a Lie-type group $G$ whose Lie algebra is $\G{g}$. 

Given any formal Poisson (algebraic) group $G$, such a Hopf algebra (over a fied of characteristic zero) was introduced in \cite{Gava02}, and then in \cite[Subsect. 1.1]{CiccGava06} as $F[[G]]$ through its duality with $U(\G{g})$, and was referred as the \emph{algebra of regular functions}. In \cite[Subsect. 1.1]{Gava07} on the other hand, given any commutative Hopf algebra $H$, over any fixed field $k$ of any characteristic, its \emph{maximal spectrum} $G$ associated to it was referred as algebraic group, and the Hopf algebra $H$ itself was called the algebra of regular functions over $G$, and is denoted by $F[G]$. A connection between these two function algebras, in the case $G$ is an affine algebraic group with $\G{g}$ being its tangent Lie algebra, may be found in \cite{Gava02}. Since we shall make use only of the duality of these algebras of functions with their corresponding universal enveloping algebras, we refer the reader to \cite{CiccGava06,Gava02,Gava07,GavaRaki09} for further details.

Hopf algebras of functions over groups, admitting pairings with universal enveloping algebras were also presented in \cite[Ex. 3]{KlimSchm-book} as \emph{coordinate Hopf algebras} of simple matrix groups. More precisely, in the case $G$ is one of the groups $SL(n,\B{C})$, $SO(n,\B{C})$, or $Sp(n,\B{C})$, the coordinate Hopf algebra, denoted by $\C{O}(G)$, was presented explicitly, and the (non-degenerate) duality between $\C{O}(G)$ and $U(\G{g})$ - $\G{g}$ being the Lie algebra of the Lie group $G$ - was remarked in \cite[Ex. 6]{KlimSchm-book}.

With a slight abuse of notation we shall denote the (Hopf) algebras of functions mentioned in the above paragraphs by $\C{O}(G)$, $G$ standing for the relevant Lie or algebraic group, whose Lie algebra being $\G{g}$. The former type of function algebras will be relevant to the content of Section \ref{sect-h-adic-vanEst} below, while those of the latter type will be appropriate for Section \ref{sect-q-adic-vanEst}. In either case, though, we shall fix the ground field to be the field of complex numbers.

We shall denote the (non-degenerate) function algebra - universal enveloping algebra pairing by
\begin{equation}\label{U(g)-O(G)-duality}
\langle\,,\,\rangle: U(\G{g})\ot \C{O}(G) \to \B{C}.
\end{equation}
Finally, we let $M$ be a right/left SAYD module over $U(\G{g})$, in such a way that the $\G{g}$-action may be integrated into a $G$-action. 

Let, now, $K\subseteq G$ be a maximal compact subgroup with Lie algebra $\G{k}\subseteq \G{g}$, and $C(U(\G{g}), U(\G{k}),M)$ be the (relative) cocyclic complex recalled in Subsection \ref{subsect-mod-coalg-cohomol}, of the quotient coalgebra $\C{C}:=U(\G{g})\ot_{U(\G{k})} \B{C}$ of \eqref{C-quotient-coalg}. In view of the duality \eqref{U(g)-O(G)-duality} then the quotient coalgebra $\C{C}:=U(\G{g})\ot_{U(\G{k})} \B{C}$ dualizes into a subalgebra of $\C{O}(G)$ that we shall denote \footnote[1]{The subalgebra $\C{O}(G/K) \subseteq \C{O}(G)$ is not defined as an algebra of functions over $G/K$, though we refer the reader to \cite[Subsect. 1.6]{CiccGava06} for an inspiration.} by $\C{O}(G/K)$. Accordingly, the quotient space $M \ot_{U(\G{g})} \C{C}^{\ot\,n}$ dualizes into the subspace $(M^\vee \ot \C{O}(G/K)^{\ot\,n})^G$. We, thus, arrive at a cyclic complex 
\[
C(\C{O}(G/K),M^\vee ) = \bigoplus_{n\geq 0} C_n(\C{O}(G/K),M^\vee ), \qquad C_n(\C{O}(G/K),M^\vee ) := (M^\vee \ot\, \C{O}(G/K)^{\ot\,n+1} )^G
\]
where $M^\vee :=\Hom(M,\B{C})$ is the corresponding SAYD contramodule over $U(\G{g})$, given by the faces
\begin{align*}
& \d_i:C_n(\C{O}(G/K),M^\vee ) \to C_{n-1}(\C{O}(G/K),M^\vee ), \qquad 0\leq i \leq n, \\
& \d_i (f\ot a_0\odots a_n) := f\ot a_0\odots a_ia_{i+1} \odots a_n, \qquad 0\leq i \leq n-1, \\
& \d_n (f\ot a_0\odots a_n) := \a(f\ot S^{-1}(\_)(a_n)a_0 \ot a_1\odots a_{n-1}),
\end{align*}
the degeneracies
\begin{align*}
& \s_j:C_n(\C{O}(G/K),M^\vee ) \to C_{n+1}(\C{O}(G/K),M^\vee ), \qquad 0\leq j \leq n, \\
& \s_j (f\ot a_0\odots a_n) := f\ot a_0\odots a_j \ot 1 \ot a_{j+1} \odots a_n, 
\end{align*}
and the cyclic operator
\begin{align*}
& \tau_n:C_n(\C{O}(G/K),M^\vee ) \to C_n(\C{O}(G/K),M^\vee ),  \\
& \tau_n (f\ot a_0\odots a_n) := \a(f\ot  S^{-1}(\_)(a_n)\ot a_0\odots a_{n-1}).
\end{align*}

The complex $C(\C{O}(G/K),M^\vee )$ defined is the Hopf-cyclic complex of the $O(G)$-comodule algebra $\C{O}(G/K)$, with coefficients in the left/right SAYD module $M^\vee$ over $\C{O}(G)$, in disguise. 

In an attempt to shed light to the relevance, we begin with the following identification.

\begin{proposition}
For any $n\geq 0$,
\[(M^\vee \ot \C{O}(G/K)^{\ot\,n+1} )^G = M^\vee \,\square_{\C{O}(G)} \,\C{O}(G/K)^{\ot\,n+1}.\]
\end{proposition}

\begin{proof}
We simply note that $f\ot a_0\odots a_n \in (M^\vee \ot \C{O}(G/K)^{\ot\,n+1} )^G$ if and only if 
\[
(f\ot a_0\odots a_n)\lt x = x^{-1}\rt f\ot (a_0\odots a_n)\lt x = f\ot a_0\odots a_n,
\]
or equivalently,
\[
f\ot (a_0\odots a_n)\lt x = (x\rt f)\ot a_0\odots a_n,
\]
for any $x\in G$. Accordingly, 
\[
f\ot (a_0^{\ps{-1}}\ldots a_n^{\ps{-1}})(x)(a_0^{\ps{0}}\odots a_n^{\ps{0}}) = f^{\ns{0}} f^{\ns{1}}(x) \ot a_0\odots a_n,
\]
for any $x\in G$, and hence
\[
f\ot a_0^{\ps{-1}}\ldots a_n^{\ps{-1}}\ot (a_0^{\ps{0}}\odots a_n^{\ps{0}}) = f^{\ns{0}} \ot f^{\ns{1}} \ot a_0\odots a_n,
\]
namely $f\ot a_0\odots a_n \in M^\vee \,\square_{\C{O}(G)} \,\C{O}(G/K)^{\ot\,n+1}$.
\end{proof}

As for the coefficient space, let $M^\vee$ be the left/right SAYD contra-module over $U(\G{g})$. Then, along the lines of \cite[Chpt. 4]{Milne-book}, the left $G$-action on $M^\vee $, given by
\[
G\times M^\vee \to M^\vee ,\qquad (x\rt f)(m) := f(m\cdot x)
\]
for any $m\in M$, any $x\in G$, and any $f\in M^\vee $, gives rise to a right $\C{O}(G)$-comodule structure 
\begin{equation}\label{O(G)-coaction-on-CM}
\nb:M^\vee \to M^\vee \ot \C{O}(G), \qquad f\mapsto f^{\ns{0}}\ot f^{\ns{1}},
\end{equation}
on $M^\vee $ through
\[
f^{\ns{0}} f^{\ns{1}}(x) = x\rt f
\]
for any $x\in G$, and any $f\in M^\vee $. Similarly, in view of the duality \eqref{U(g)-O(G)-duality}, the quotient coalgebra $\C{C}=U(\G{g})\ot_{U(\G{k})} \B{C}$ being a left $U(\G{g})$-module coalgebra, the subalgebra $\C{O}(G/K)\subseteq \C{O}(G)$ is a right $U(\G{g})$-module algebra (see for instance \cite[Prop. 1.6.19]{Majid-book}), and hence a left $O(G)$-comodule algebra (see for instance \cite[Prop. 1.6.11]{Majid-book}) which we shall denote by
\begin{equation}\label{O(G)-coaction-on-O(G/K)}
\Db:\C{O}(G/K)\to \C{O}(G)\ot \C{O}(G/K), \qquad a\mapsto a^{\ps{-1}}\ot a^{\ps{0}}.
\end{equation}

On the other hand, it follows from the duality \eqref{U(g)-O(G)-duality} that the left $U(\G{g})$-coaction on $M$ gives rise to a right $\C{O}(G)$-action via
\begin{equation}\label{right-OG-action}
m\lt \a = \langle m\ns{-1}, \a\rangle m\ns{0},
\end{equation}
for any $m\in M$, and any $\a\in\C{O}(G)$, and hence a left $\C{O}(G)$-action
\begin{equation}\label{O(G)-action-on-CM}
\rt:\C{O}(G)\ot M^\vee \to M^\vee , \qquad (\a\rt f)(m) := f(m\lt \a)
\end{equation}
on $M^\vee $.

\begin{proposition}\label{prop-SAYD-dual}
Let $M$ be a right/left SAYD module over $U(\G{g})$. The action \eqref{O(G)-action-on-CM}, and the coaction \eqref{O(G)-coaction-on-CM} endows $M^\vee $ with the structure of a left/right SAYD module over $\C{O}(G)$.
\end{proposition}

\begin{proof}
Let us first consider the left/right AYD compatibility, for the details of which we refer the reader to \cite{HajaKhalRangSomm04-I}. For any $a\in \C{O}(G/K)$, any $x\in G$, any $m\in M$, and any $f\in M^\vee $,
\begin{align*}
& (a\rt f)^{\ns{0}}(m) (a\rt f)^{\ns{1}} (x) = (a\rt f)(m\cdot x) = \\
& f((m\cdot x)\lt a) = \langle (m\cdot x)\ns{-1},a\rangle f((m\cdot x)\ns{0}) =  \\
& \langle \Ad_{x^{-1}}(m\ns{-1}),a\rangle f(m\cdot x) = \langle m\ns{-1},\Ad_x(a)\rangle f(m\cdot x) = \\
& \langle m\ns{-1},a\ps{2}\rangle a\ps{1}(x^{-1})a\ps{3}(x) f(m\cdot x) = (a\ps{2}\rt f^{\ns{0}})(m) (a\ps{3}f^{\ns{1}}S(a\ps{1}))(x),
\end{align*}
where 
\[
(m\cdot x)\ns{-1} \ot (m\cdot x)\ns{0} = \Ad_{x^{-1}}(m\ns{-1}) \ot m\ns{0}\cdot x
\]
is the integration of the $U(\G{g})$-AYD compatibility on $M$. As a result,
\[
\nb(a\rt f) = a\ps{2}\rt f^{\ns{0}} \ot a\ps{3}f^{\ns{1}}S(a\ps{1}).
\]
As for the stability, we see at once that
\begin{align*}
& (f^{\ns{1}}\rt f^{\ns{0}} )(m)= f^{\ns{0}} (m\lt f^{\ns{1}} ) = \langle m\ns{-1},f^{\ns{1}} \rangle f^{\ns{0}} (m\ns{0}) = \\
& f(m\ns{0}m\ns{-1}) = f(m),
\end{align*}
for any $m\in M$, and any $f\in M^\vee $. Therefore,
\[
f^{\ns{1}}\rt f^{\ns{0}}  = f.
\]
\end{proof}

\begin{remark}\label{rk-contra}
On the other extreme, the right/left SAYD module $M$ over $U(\G{g})$ may be considered as a right/left SAYD contra-module over $\C{O}(G)$. Furthermore, $M$ has even the structure of a right/left SAYD module over $\C{O}(G)$ by the right action \eqref{right-OG-action}, and the left $\C{O}(G)$-coaction 
\[
M\to \C{O}(G) \ot M, \qquad m\mapsto m^{\ns{-1}} \ot m^{\ns{0}}
\]  
given by
\[
\langle h, m^{\ns{-1}} \rangle m^{\ns{0}} = mh
\]
for any $h\in U(\G{h})$.
\end{remark}

Accordingly, the complex $C(\C{O}(G/K),M^\vee )$ may be realized as
\[
C(\C{O}(G/K),M^\vee ) = \bigoplus_{n\geq 0} C_n(\C{O}(G/K),M^\vee ), \qquad C_n(\C{O}(G/K),M^\vee ) := M^\vee \, \square_{\C{O}(G)}\, \C{O}(G/K)^{\ot\,n+1} 
\]
with the faces
\begin{align}\label{face-classical-Hopf}
\begin{split}
& \d_i:C_n(\C{O}(G/K),M^\vee ) \to C_{n-1}(\C{O}(G/K),M^\vee ), \qquad 0\leq i \leq n, \\
& \d_i (f\ot a_0\odots a_n) := f\ot a_0\odots a_ia_{i+1} \odots a_n, \qquad 0\leq i \leq n-1, \\
& \d_n (f\ot a_0\odots a_n) := (a_n^{\ps{-1}}\rt f)\ot a_n^{\ps{0}}a_0 \ot a_1\odots a_{n-1},
\end{split}
\end{align}
the degeneracies
\begin{align}\label{degeneracy-classical-Hopf}
\begin{split}
& \s_j:C_n(\C{O}(G/K),M^\vee ) \to C_{n+1}(\C{O}(G/K),M^\vee ), \qquad 0\leq j \leq n, \\
& \s_j (f\ot a_0\odots a_n) := f\ot a_0\odots a_j \ot 1 \ot a_{j+1} \odots a_n, 
\end{split}
\end{align}
and the cyclic operator
\begin{align}\label{cyclic-classical-Hopf}
\begin{split}
& \tau_n:C_n(\C{O}(G/K),M^\vee ) \to C_n(\C{O}(G/K),M^\vee ),  \\
& \tau_n (f\ot a_0\odots a_n) := (a_n^{\ps{-1}}\rt f)\ot a_n^{\ps{0}}\ot a_0 \ot a_1\odots a_{n-1},
\end{split}
\end{align}
which is nothing but the Hopf-cyclic homology complex of the $\C{O}(G)$-comodule algebra $\C{O}(G/K)$, with coefficients in the left/right SAYD module $M^\vee $ over $\C{O}(G)$. Compare with the left/left version in Subsection \ref{subsect-comod-alg-homol}, and see also \cite{HajaKhalRangSomm04-II}.


We shall denote the cyclic (resp. periodic cyclic) homology of the cyclic module above by $HC_\ast(\C{O}(G/K),M^\vee )$ (resp. $HP_\ast(\C{O}(G/K),M^\vee )$), and call it the \emph{(periodic) Hopf-cyclic homology of the $\C{O}(G)$-comodule algebra $\C{O}(G/K)$, with coefficients in the SAYD module $M^\vee $ over $\C{O}(G)$}.

Dually, in view of Subsection \ref{subsect-contra-comod-alg}, we now consider the complex 
\begin{align*}
& C(\C{O}(G/K),M) = \bigoplus_{n\geq 0} C^n(\C{O}(G/K),M), \\
& C^n(\C{O}(G/K),M) := \Hom(M^\vee \, \square_{\C{O}(G)}\, \C{O}(G/K)^{\ot\,n+1},\B{C}) \cong M\ot_{U(\G{g})} \Hom(\C{O}(G/K)^{\ot\,n+1}, \B{C})
\end{align*}
that fit into the pairing
\begin{align*}
&\langle\,,\,\rangle: C_n(\C{O}(G/K),M^\vee ) \ot C^n(\C{O}(G/K),M)\to \B{C}, \\
&\hspace{4cm} \langle f\ot a_0\odots a_n\,, m\ot_{U(\G{g})}\phi\rangle := f(m)\phi(a_0\odots a_n),
\end{align*}
for any $\phi \in \Hom(\C{O}(G/K)^{\ot\,n+1}, M)$. Accordingly, the cyclic structure on $C(\C{O}(G/K),M^\vee )$ induces a cocyclic structure on $C(\C{O}(G/K),M)$ via the cofaces
\begin{align*}
& d_i:C^{n-1}(\C{O}(G/K),M) \to C^n(\C{O}(G/K),M), \qquad 0\leq i \leq n, \\
& (d_i (m\ot_{U(\G{g})}\phi))(a_0\odots a_n) := m\phi(a_0 \odots a_ia_{i+1} \odots a_n), \qquad 0\leq i \leq n-1, \\
& (d_n (m\ot_{U(\G{g})}\phi))(a_0\odots a_n) := m\ns{0}\phi(S^{-1}(m\ns{-1})(a_n)a_0 \ot a_1 \odots a_{n-1}),
\end{align*}
the codegeneracies
\begin{align*}
& s_j:C^{n+1}(\C{O}(G/K),M) \to C^n(\C{O}(G/K),M), \qquad 0\leq j \leq n, \\
& (s_j (m\ot_{U(\G{g})}\phi))(a_0\odots a_n) := m\phi(a_0 \odots a_j\ot 1 \ot a_{j+1} \odots a_n), 
\end{align*}
and the cyclic operator
\begin{align*}
& t_n:C^n(\C{O}(G/K),M) \to C^n(\C{O}(G/K),M),  \\
& (t_n (m\ot_{U(\G{g})}\phi))(a_0\odots a_n) := m\ns{0}\phi(S^{-1}(m\ns{-1})(a_n)\ot a_0 \ot a_1 \odots a_{n-1}).
\end{align*}

Let us record also that (as was done in Subsection \ref{subsect-contra-comod-alg}) the cyclic complex $C(\C{O}(G/K),M)$ computes the Hopf-cyclic cohomology of the $\C{O}(G)$-comodule algebra $\C{O}(G/K)$, with coefficients in the right/left SAYD contra-module $M$ over $\C{O}(G)$. Indeed, the isomorphism
\[
M\ot_{U(\G{g})} \Hom(\C{O}(G/K)^{\ot\,n+1}, \B{C}) \cong \Hom(M^\vee \,\square_{\C{O}(G)}\,\C{O}(G/K)^{\ot\,n+1}, \B{C})
\]
allows us to reorganize the complex $C(\C{O}(G/K),M)$ with the cofaces
\begin{align}\label{coface-classical-Hopf}
\begin{split}
& d_i:C^{n-1}(\C{O}(G/K),M) \to C^n(\C{O}(G/K),M), \qquad 0\leq i \leq n, \\
& (d_i (m\ot_{U(\G{g})}\phi))(a_0\odots a_n) := m\phi(a_0 \odots a_ia_{i+1} \odots a_n), \qquad 0\leq i \leq n-1, \\
& (d_n (m\ot_{U(\G{g})}\phi))(a_0\odots a_n) := (m\lt a_n^{\ps{-1}})\phi(a_n^{\ps{0}}a_0 \ot a_1 \odots a_{n-1}),
\end{split}
\end{align}
the codegeneracies
\begin{align}\label{codegeneracy-classical-Hopf}
\begin{split}
& s_j:C^{n+1}(\C{O}(G/K),M) \to C^n(\C{O}(G/K),M), \qquad 0\leq j \leq n, \\
& (s_j (m\ot_{U(\G{g})}\phi))(a_0\odots a_n) := m\phi(a_0 \odots a_j\ot 1 \ot a_{j+1} \odots a_n), 
\end{split}
\end{align}
and the cyclic operator
\begin{align}\label{cocyclic-classical-Hopf}
\begin{split}
& t_n:C^n(\C{O}(G/K),M) \to C^n(\C{O}(G/K),M),  \\
& (t_n (m\ot_{U(\G{g})}\phi))(a_0\odots a_n) := (m\lt a_n^{\ps{-1}})\phi(a_n^{\ps{0}}\ot a_0 \ot a_1 \odots a_{n-1}),
\end{split}
\end{align}
which clearly is the complex computing the Hopf-cyclic cohomology of the $\C{O}(G)$-comodule algebra $\C{O}(G/K)$, with coefficients in the right/left SAYD contra-module $M$ over $\C{O}(G)$.

Similarly to the cyclic case above, we shall denote the cyclic (resp. periodic cyclic) homology of the above cocyclic module by $HC^\ast(\C{O}(G/K),M)$ (resp. $HP^\ast(\C{O}(G/K),M)$), and we shall refer to it as the \emph{(periodic) Hopf-cyclic cohomology of the $\C{O}(G)$-comodule algebra $\C{O}(G/K)$, with coefficients in the SAYD contra-module $M$ over $\C{O}(G)$}.

\subsection{The van Est isomorphism}\label{subsect-vanEst-classical-Hopf-alg}~

\begin{proposition}\label{prop-vanEst-classical-I}
Let $\C{O}(G)$ be the Hopf algebra of functions over the group $G$, and let $\G{g}$ denote the Lie algebra of $G$. Let also $K\subseteq G$ be a maximal compact subgroup, the Lie algebra of which being $\G{k}$. Furthermore, let $M$ be a right/left SAYD module over $U(\G{g})$, so that the $\G{g}$-action may be integrated into a $G$-action, and $M^\vee  = \Hom(M,\B{C})$ the corresponding left/right SAYD contra-module. Then,
\begin{equation}\label{isom-E_1-cohom}
HC_\ast(U(\G{g}),U(\G{k}),M^\vee ) \cong HC_\ast(\C{O}(G/K),M^\vee ).
\end{equation}
\end{proposition}

\begin{proof}
It follows from \cite[Lemma 6.2]{JaraStef06} that since $M$ is an AYD module over $U(\G{g})$, it admits a (bounded) increasing filtration $(F_pM)_{p\in \B{Z}}$, so that $F_pM / F_{p-1}M$ is a trivial $U(\G{g})$-comodule for any $p\in \B{Z}$. Accordingly, there is a decreasing filtration on $M^\vee $ by $F_pM^\vee :=\Hom_k(F_pM,\B{C})$, which satisfies 
\[
\frac{F_{p-1}M^\vee }{F_pM^\vee } \cong \left(\frac{F_pM}{F_{p-1}M}\right)^\vee.
\]
We shall now compare the complexes $C(U(\G{g}),U(\G{k}),F_{p-1}M^\vee /F_pM^\vee )$ and $C(\C{O}(G/K),F_{p-1}M^\vee /F_pM^\vee )$ in the $E_1$-level of the associated spectral sequences. 

As is given in the proof of \cite[Thm. 15]{ConnMosc}, the Hochschild boundary of the former coincides with zero map, while the Connes coboundary operator corresponds to the Lie algebra (Chevalley-Eilenberg) cohomology coboundary. 

As for the latter, it follows from the Hochschild-Kostant-Rosenberg theorem that its Hochschild homology classes may be identified with the space $A(G/K,F_{p-1}M^\vee /F_pM^\vee )^G$ of $(F_{p-1}M^\vee /F_pM^\vee )$-valued invariant differential forms\footnote[1]{The Hochschild-Kostant-Rosenberg map commutes with the (diagonal) $G$-action.}. Moreover, it is also known that the Connes coboundary operator corresponds, on these Hochschild classes, the exterior derivative of differential forms. On the other hand, since $K\subseteq G$ is a maximal compact subgroup, it is also known that $G/K$ is diffeomorphic to an Euclidean space. As such,
\[
\xymatrix{
0\ar[r] &M^\vee \ar[r] & A^0(G/K, F_{p-1}M^\vee /F_pM^\vee) \ar[r]^{d} & A^1(G/K, F_{p-1}M^\vee /F_pM^\vee) \ar[r]^{\,\,\,\,\,\,\,\,\,\,\,\,\,\,\,\,\,\,\,\,\,\,\,\,\,\,\,\,\,\,\,\,\,\,\,\,d} & \ldots 
}
\]
is an (continuously) injective resolution of $M^\vee $, see \eqref{resolution-Lie-alg} above, and hence the homology with respect to the Connes coboundary operator corresponds nothing but to the (continuous) group cohomology, see also \cite[Chpt. II]{Brown-book}.

Finally, the isomorphism on the level of the $E_1$-terms is given by the van Est isomorphism; see Subsection \ref{subsect-vanEst}. 
\end{proof}

Let us record next the homological counterpart of the above result. 

\begin{proposition}\label{prop-vanEst-classical-II}
Let $\C{O}(G)$ be the Hopf algebra of functions over the group $G$, and let $\G{g}$ denote the Lie algebra of $G$. Let also $K\subseteq G$ be a maximal compact subgroup, the Lie algebra of which being $\G{k}$. Furthermore, let $M$ be a right/left SAYD module over $U(\G{g})$, so that the $\G{g}$-action may be integrated into a $G$-action. Then,
\begin{equation}\label{isom-E_1}
HC^\ast(U(\G{g}),U(\G{k}),M) \cong HC^\ast(\C{O}(G/K),M).
\end{equation}
\end{proposition}

\begin{proof}
It follows from \cite[Lemma 6.2]{JaraStef06} that since $M$ is an AYD module over $U(\G{g})$, it admits a (bounded) increasing filtration $(F_pM)_{p\in \B{Z}}$, so that $F_pM / F_{p-1}M$ is a trivial $U(\G{g})$-comodule for any $p\in \B{Z}$. Accordingly, we compare the complexes $C(U(\G{g}),U(\G{k}),F_pM / F_{p-1}M) $ and $C(\C{O}(G/K),F_pM / F_{p-1}M)$ in the $E_1$-level of the associated spectral sequences. 

As is given in the proof of \cite[Thm. 15]{ConnMosc}, the Hochschild coboundary of the former coincides with the zero map, while the Connes boundary operator corresponds to the Lie algebra (Chevalley-Eilenberg) homology boundary. 

As for the latter complex, dually to the previous proposition (see also \cite[Lemma 45(a)]{Conn85}) the Hochschild cohomology classes may be identified with the $G$-coinvariants $[C_\ast(G/K,M)]_G$ of de Rham currents, with coeffcients in $M$. It then follows from \cite[Lemma 45(b)]{Conn85} that on these Hochschild classes the Connes boundary map operates as de Rham boundary for currents. Once again, since $G/K$ is diffeomorphic to an Euclidean space, 
\begin{equation}\label{G-homology}
\xymatrix{
\ldots\ar[r] & C_1(G/K,M) \ar[r]^{\p_{dR}} & C_0(G/K,M) \ar[r]^{\,\,\,\,\,\,\,\,\,\,\,\,\,\,\,\p_{dR}}\ar[r] &M\ar[r] & 0 
}
\end{equation}
is a (continuously) projective resolution of $M$. As such, its coinvariants compute the group homology, \cite[Chpt. II]{Brown-book}.

Finally, the desired isomorphism is induced by the van Est isomorphism on the dual picture, in view of \cite[Thm. 2]{BlanWign83} and Corollary \ref{BlancWigner-Lie-alg-cyclic-version}.
\end{proof}

We can further state the following extension of \cite[Thm. 2]{BlanWign83}.

\begin{corollary}
Let $\C{O}(G)$ be the Hopf algebra of functions over the $G$, and let $N$ be a left/right SAYD module over $\C{O}(G)$. Let also $N^\vee = \Hom(N,\B{C})$ denote the corresponding right/left SAYD contra-module over $\C{O}(G)$. Then,
\[
HC^\ast(\C{O}(G),N^\vee)^\vee \cong HC_\ast(\C{O}(G),N).
\]
\end{corollary}

\section{The van Est isomorphism on quantized Hopf algebras}\label{sect-h-adic-vanEst}

In this section we shall develop the van Est isomorphisms for the $h$-adic quantum groups, using a natural $h$-filtration on their Hopf-cyclic complexes.

\subsection{Quantized Hopf algebras}\label{subsect-quaintized-Hopf-alg}~

In the present subsection we shall recall the quantized Hopf algebras on which the van Est isomorphism will be considered. Namely, we shall take a quick overview of the quantized universal enveloping algebras and the quantized function algebras.

Let us recall from \cite[Def. 3.10]{FengTsyg91} and \cite[Def. 6.2.4]{ChariPressley-book}, see also \cite[Def. 1.2(b)]{Gava02}, that a quantization of a Poisson Hopf algebra $A_0$ over $\B{C}$ is a (topological, with respect to the $h$-adic topology) Hopf algebra $A$ over $\B{C}[[h]]$ such that $A/hA$ is isomorphic, as Poisson Hopf algebras, to $A_0$.

Motivated by \cite[Thm. 3.13]{FengTsyg91} and \cite[Def. 6.2.4]{ChariPressley-book}, given a Poisson-Lie group $G$, we shall denote a quantization of $\C{O}(G)$ by $\C{O}_h(G)$, and we shall call it the \emph{quantized algebra of functions} over $G$, or simply the \emph{quantized function algebra}.

In particular, following \cite[Sect. 7.1]{ChariPressley-book} - see also \cite[Ex. 3.15]{FengTsyg91}, the quantized function algebra $\C{O}_h(SL_2(\B{C}))$ is the topological Hopf algebra (over $\B{C}[[h]]$) given by
\begin{align*}
& ac=e^{-h}ca, \qquad bd = e^{-h}db, \qquad ab=e^{-h}ba, \qquad cd=e^{-h}dc, \\
& bc=cb, \qquad ad-da=(e^{-h}-e^h)bc,\qquad ad-e^{-h}bc = 1,\\
& \D(a)=a\widehat{\ot}a + b \widehat{\ot} c, \qquad \D(b)=a\widehat{\ot} b + b \widehat{\ot} d, \\
& \D(c)=c\widehat{\ot}a + d \widehat{\ot} c, \qquad \D(d)=c\widehat{\ot} b + d \widehat{\ot} d,\\
& \ve(a) = \ve (d) = 1, \qquad \ve(b) = \ve(c) = 0,\\
& S(a) = d, \quad S(b) = -e^{h}b, \quad S(c)=-e^{-h}c,\quad S(d)=a.
\end{align*}

Dually, a quantization of a co-Poisson Hopf algebra $H_0$ over $\B{C}$ is a (topological, with respect to the $h$-adic topology) Hopf algebra $H$ over $\B{C}[[h]]$ such that $H/hH$ is isomorphic, as co-Poisson Hopf algebras, to $H_0$.

As was noted in \cite[Thm. 3.11]{FengTsyg91}, given a Lie bialgebra $\G{g}$, the universal enveloping algebra $U(\G{g})$ has a unique quantization, called the \emph{quantized universal enveloping algebra}, which is denoted by $U_h(\G{g})$.

In particular, the basic example corresponding to the Lie bialgebra structure on $s\ell_2(\B{C})$, which is given by
\begin{align*}
& \d:s\ell_2(\B{C}) \to s\ell_2(\B{C})\wedge s\ell_2(\B{C}),\\
& \d(H) := 0, \qquad \d(E) := E\wedge H, \qquad \d(F) := F\wedge H.
\end{align*}
is given in \cite[Sec. 6.4]{ChariPressley-book}. More precisely, as was presented in \cite[Def.-Prop. 6.4.3]{ChariPressley-book} - see also \cite[Ex. 3.15]{FengTsyg91} and \cite[Subsect. 3.1.5]{KlimSchm-book} - the quantized universal enveloping algebra $U_h(s\ell_2(\B{C}))$ is the topological Hopf algebra (over $\B{C}[[h]]$) given by
\begin{align*}
& [H,E]=2E, \qquad [H,F] = -2F, \qquad [E,F] = \frac{e^{hH}-e^{-hH}}{e^h-e^{-h}}, \\
& \D(H) = 1\widehat{\ot} H + H \widehat{\ot} 1, \quad \D(E) = E\widehat{\ot} e^{hH} + 1 \widehat{\ot} E, \quad \D(F) = F\widehat{\ot} 1 + e^{-hH} \widehat{\ot} F, \\
& \ve(H) = \ve(E) = \ve(F) = 0,\\
& S(H) = -H, \qquad S(E)=-Ee^{-hH}, \qquad S(F)=-e^{hH}F.
\end{align*}

Finally, along the lines of Subsection \ref{subsect-coord-alg-funct}, we denote by $\C{O}_h(G/K) \subseteq \C{O}_h(G)$ the dual subalgebra of the quotient coalgebra $\C{C}_h:=U_h(\G{g})\ot_{U_h(\G{k})} \B{C}[[h]]$, in view of the duality between $\C{O}_h(G)$ and $U_h(\G{g})$, where $K \subseteq G$ is a maximal compact subgroup, and $\G{k}$ stands for the Lie algebra of $K$. 

\subsection{Quantized van Est isomorphism}\label{subsect-quantized-vanEst}~

Following the terminology of \cite[Sect. 1]{Gava02}, we shall mean by a $\B{C}[[h]]$-module; a torsionless, complete, and separated $\B{C}[[h]]$-module, equipped with the $h$-adic topology. Accordingly, if $V$ is a $\B{C}[[h]]$-module, then as was noted in \cite[Subsect. 1.1]{Gava02}, we have $V \cong V_0[[h]]$ as $k[[h]]$-modules, where $V_0:=V/hV$ is the \emph{semi-classical limit} of $V$.

Let, now, $V$ be a (right) module over a quantized Hopf algebra $P$; that is, $V$ is a $k[[h]]$-module equipped with a $\B{C}[[h]]$-linear (hence, continuous) map
\begin{equation}\label{P-action-quantized}
\rt: V \widehat{\ot} P \to V
\end{equation}
satisfying the usual compatibilities for a module.Let us note that the tensor product refers to the completed tensor product over $\B{C}[[h]]$. Tensoring both sides with $\B{C}$ over $\B{C}[[h]]$, then, renders a linear map
\begin{equation}\label{P-action-semiclassical}
\rt: V_0 \ot P_0 \to V_0,
\end{equation}
which also satisfies the module compatibilities. That is, the $\B{C}$-module $V_0$ is then a module over the (Poisson, or co-Poisson) Hopf algebra $P_0:=P/P_0$.

Similarly, a (left) comodule $V$ over a quantized Hopf algebra $P$ is a $\B{C}[[h]]$-module equipped with a $\B{C}[[h]]$-linear map
\begin{equation}\label{P-coaction-quantized}
\nb:V\to P\widehat{\ot} V
\end{equation}
that satisfies the comodule compatibilities. Similarly, the application of $\ot_{\B{C}[[h]]}\,\B{C}$ yields
\begin{equation}\label{P-coaction-semiclassical}
\nb:V_0\to P_0\ot V_0
\end{equation}
satisfying the comodule compatibilities.

Along the lines above, we define a (right/left) SAYD module $V$ over a quantized Hopf algebra $P$ as a $\B{C}[[h]]$-module equipped with a right $P$-action as \eqref{P-action-quantized}, and a left $P$-coaction as \eqref{P-coaction-quantized}, so that the usual SAYD compatibilities are satisfied. Then, similarly above, the semi-classical limit $V_0$ of $V$ happens to be the SAYD module over the (Poisson, or co-Poisson) Hopf algebra $P_0$ through \eqref{P-action-semiclassical} and \eqref{P-coaction-semiclassical}.

Finally, we have the Hopf-cyclic complexes associated to the quantized Hopf algebras and SAYD modules over them, though in the presence of the topological tensor products and the $\B{C}[[h]]$-linear (continuous) maps.\footnote[1]{We refer the reader to \cite{RangSutl-V} for further details on Hopf-cyclic cohomology for topological Hopf algebras.} More precisely, $M$ being a right/left SAYD module over $U_h(\G{g})$, we dualize the (relative) Hopf-cyclic complex $C(U_h(\G{g}), U_h(\G{k}),M)$ in the sense of Subsection \ref{subsect-mod-coalg-cohomol} to get a cyclic complex (with coefficients in the )
\[
C(\C{O}_h(G/K),M^\vee) = \bigoplus_{n\geq 0} C_n(\C{O}_h(G/K),M^\vee), \qquad C_n(\C{O}_h(G/K),M^\vee) := M^\vee \square_{\C{O}(G)} \,\C{O}_h(G/K)^{\widehat{\ot}\, n+1}
\]
whose face, degeneracy, and cyclic operators are the same as \eqref{face-classical-Hopf}, \eqref{degeneracy-classical-Hopf}, and \eqref{cyclic-classical-Hopf} respectively.

We shall denote the cyclic (resp. periodic cyclic) homology of this complex by $HC_\ast(\C{O}_h(G/K),M^\vee )$ (resp. $HP_\ast(\C{O}_h(G/K),M^\vee )$), and call it the \emph{(periodic) Hopf-cyclic homology of the $\C{O}_h(G)$-comodule algebra $\C{O}_h(G/K)$, with coefficients in the SAYD module $M^\vee $ over $\C{O}_h(G)$}.

Along the lines of Subsection \ref{subsect-coord-alg-funct}, viewing $M$ as a right/left SAYD contra-module over $\C{O}_h(G)$, we then also have the complex
\begin{align*}
& C(\C{O}_h(G/K),M) = \bigoplus_{n\geq 0} C^n(\C{O}_h(G/K),M), \\
& C^n(\C{O}_h(G/K),M) := \Hom(M^\vee \, \square_{\C{O}_h(G)}\, \C{O}_h(G/K)^{\widehat{\ot}\,n+1},\B{C}) \cong M\widehat{\ot}_{U_h(\G{g})} \Hom(\C{O}_h(G/K)^{\widehat{\ot}\,n+1}, \B{C})
\end{align*}
whose cofaces, codegeneracies, and the cyclic operator are given just as \eqref{coface-classical-Hopf}, \eqref{codegeneracy-classical-Hopf}, and \eqref{cocyclic-classical-Hopf} respectively.

We shall denote the cyclic (resp. periodic cyclic) homology of this cocyclic module by $HC^\ast(\C{O}(G/K),M)$ (resp. $HP^\ast(\C{O}(G/K),M)$), and call it the \emph{(periodic) Hopf-cyclic cohomology of the $\C{O}(G)$-comodule algebra $\C{O}(G/K)$, with coefficients in the SAYD contra-module $M$ over $\C{O}(G)$}.

We are now ready for the van Est isomorphisms on the level of quantized Hopf algebras. This time, we begin with the homological one.

\begin{theorem}\label{vanEst-I-cohom}
Let $G$ be a Poisson-Lie group, with the quantized function algebra $\C{O}_h(G)$, and the Lie (bi)algebra $\G{g}$. Let also $K\subseteq G$ be a maximal compact subgroup, the Lie algebra of which being $\G{k}$. Furthermore, let $M$ be a right/left SAYD module over $U_h(\G{g})$, so that the $\G{g}$-action may be integrated into a $G$-action. Then,
\[
HC^\ast(U_h(\G{g}),U_h(\G{k}),M) \cong HC^\ast(\C{O}_h(G/K),M).
\]
\end{theorem}

\begin{proof}
Let us consider the decreasing filtrations on both complexes through $h$, that is,
\[
F_p C^\ast(U_h(\G{g}),U_h(\G{k}),M) := h^p C^\ast(U_h(\G{g}),U_h(\G{k}),M), \qquad p\geq 0,
\]
with $F_p C^\ast(U_h(\G{g}),U_h(\G{k}),M) := 0$ for $p<0$, and
\[
F_p C^\ast(\C{O}_h(G/K),M) := h^p C^\ast(\C{O}_h(G/K),M), \qquad p\geq 0,
\]
with $F_p C^\ast(\C{O}_h(G/K),M) := 0$ for $p<0$. 

It is evident by the $\B{C}[[h]]$-linearity of the (total) differential maps that both  $C^\ast(U_h(\G{g}),U_h(\G{k}),M)$ and $C^\ast(\C{O}_h(G/K),M)$ becomes filtered complexes through these filtrations.

On the other hand, both filtrations are clearly not (necessarily) bounded. Nevertheless, they both are \emph{weakly convergent} in the sense of \cite[Def. 3.1]{McCleary-book}, that is,  
\[
Z^{i,j}_\infty = \cap_r\,Z_r^{i,j},
\]
where, referring the differential maps simply as $d:C^n\to C^{n+1}$, here $Z_r^{i,j}:= F^iC^{i+j} \cap d^{-1} (F^{i+r}C^{i+j+1})$, and $Z_\infty^{i,j}:=F^iC^{i+j} \cap \ker(d)$. This, more precisely, follows from the finiteness (of the Hochschild cohomology classes) on the columns of the associated bicomplexes 
\[
E_0^{i,j}(U_h(\G{g}),U_h(\G{k}),M) := \frac{F^iC^{i+j}(U_h(\G{g}),U_h(\G{k}),M)}{F^{i+1}C^{i+j}(U_h(\G{g}),U_h(\G{k}),M)} \cong h^iC^{j+i}(U(\G{g}),U(\G{k}),M_0),
\]
and
\[
E_0^{i,j}(\C{O}_h(G/K),M) := \frac{F^iC^{i+j}(\C{O}_h(G/K),M)}{F^{i+1}C^{i+j}(\C{O}_h(G/K),M)} \cong h^iC^{j+i}(\C{O}(G/K),M_0).
\]
As a result of \cite[Thm. 3.2]{McCleary-book}, the corresponding spectral sequences converge in the level of Hochschild cohomology, and hence in the level of the cyclic cohomology.

Furthermore, the induced maps $d_0:E^{i,j}_0\to E^{i,j+1}_0$ correspond to the (total) Hopf-cyclic differential maps on the semi-classical limits of the individual complexes. Finally, an isomorphism on the level of $E_1$-terms is given by \eqref{isom-E_1}.
\end{proof}

The cohomological counterpart of the van Est isomorphism on the quantized Hopf algebras, whose proof is omitted due to its similarity to Proposition \ref{vanEst-I-cohom}, is given below.

\begin{theorem}\label{vanEst-II-cohom}
Let $G$ be a Poisson-Lie group, with the quantized function algebra $\C{O}_h(G)$, and the Lie (bi)algebra $\G{g}$. Let also $K\subseteq G$ be a maximal compact subgroup, the Lie algebra of which being $\G{k}$. Furthermore, let $M$ be a right/left SAYD module over $U_h(\G{g})$, so that the $\G{g}$-action may be integrated into a $G$-action, and let $M^\vee  = \Hom_{\B{C}[[h]]}(M,\B{C}[[h]])$ be the corresponding left/right SAYD contra-module. Then,
\[
HC_\ast(U_h(\G{g}),U_h(\G{k}),M^\vee ) \cong HC_\ast(\C{O}_h(G/K),M^\vee ).
\]
\end{theorem}

\section{The van Est isomorphism on quantum groups}\label{sect-q-adic-vanEst}

In this final section we shall prove the $q$-adic counterparts of the Hopf-cyclic (homology and cohomology) van Est isomorphisms considered in the previous section.

\subsection{Drinfeld-Jimbo algebras}\label{subsect-Drinfeld-Jimbo}~

Let us recall from \cite[Subsect. 6.1.2]{KlimSchm-book} the quantum enveloping algebras (Drinfeld-Jimbo algebras) of Lie algebras.

To this end, let $\G{g}$ be a finite dimensional semi-simple complex Lie algebra, and let $\a_1,\ldots,\a_\ell$ be an ordered sequence of simple roots. Let also $A=[a_{ij}]$ be the Cartan matrix associated to $\G{g}$, and let $q$ be a fixed nonzero complex number such that $q_i^2\neq 1$, where $q_i:= q^{d_i}$, $1\leq i\leq \ell$, and $d_i=(\a_i,\a_i)/2$.

The algebra $U_q(\G{g})$ is defined to be the Hopf algebra with $4\ell$ generators $E_i,F_i,K_i,K_i^{-1}$, $1\leq i\leq \ell$, subject to the relations
\begin{align*}
& K_iK_j = K_jK_i,\qquad K_iK_i^{-1}=K_i^{-1}K_i=1,\\
& K_iE_jK_i^{-1}=q_i^{a_{ij}}E_j,\qquad K_iF_jK_i^{-1}=q_i^{-a_{ij}}F_j,\\
& E_iF_j-F_jE_i=\d_{ij}\frac{K_i-K_i^{-1}}{q_i-q_i^{-1}},\\
& \sum_{r=0}^{1-a_{ij}}(-1)^r\left[\begin{array}{c}
                                              1-a_{ij} \\
                                              r
                                            \end{array}
\right]_{q_i}E_i^{1-a_{ij}-r}E_jE_i^r=0, \quad i\neq j,\\
& \sum_{r=0}^{1-a_{ij}}(-1)^r\left[\begin{array}{c}
                                              1-a_{ij} \\
                                              r
                                            \end{array}
\right]_{q_i}F_i^{1-a_{ij}-r}F_jF_i^r=0, \quad i\neq j,
\end{align*}
where
\begin{equation*}
\left[\begin{array}{c}
        n \\
        r
      \end{array}
\right]_q = \frac{(n)_q\,!}{(r)_q\,!\,\,(n-r)_q\,!},\qquad (n)_q:=\frac{q^n-q^{-n}}{q-q^{-1}}.
\end{equation*}
Furthermore, the algebra $U_q(\G{g})$ may be endowed with a Hopf algebra structure via
\begin{align*}
& \D(K_i)=K_i\ot K_i,\quad \D(K_i^{-1})=K_i^{-1}\ot K_i^{-1}, \\
& \D(E_i)=E_i\ot K_i + 1\ot E_i,\quad \D(F_j)=F_j\ot 1 + K_j^{-1}\ot F_j, \\
& \ve(K_i)=1,\quad \ve(E_i)=\ve(F_i)=0,\\
& S(K_i)=K_i^{-1},\quad S(E_i)=-E_iK_i^{-1},\quad S(F_i)=-K_iF_i.
\end{align*}
Although the construction is defined for semi-simple Lie algebras, it extends to other Lie algebras such as $g\ell_n$. We recall from \cite[Subsect. 6.1.2]{KlimSchm-book} that $U_q(g\ell_n)$ is the algebra generated by $E_i,F_i$, $1\leq i \leq n-1$, along with $K_j,K_j^{-1}$, $1\leq j \leq n$, subject to the relations
\begin{align*}
& K_iK_j = K_jK_i, \qquad K_iK_i^{-1} = K_i^{-1}K_i=1, \\
& K_iE_jK_i^{-1} = q^{\d_{i,j}-\d_{i,(j+1)}}E_j, \qquad K_iF_jK_i^{-1} = q^{-\d_{i,j}+\d_{i,(j+1)}}F_j,\\
& E_iF_j-F_jE_i =\d_{ij}\frac{K_iK_{i+1}^{-1} - K_i^{-1}K_{i+1}}{q-q^{-1}}, \\
& E_iE_j = E_jE_i, \qquad F_iF_j = F_jF_i, \qquad |i-j|\leq 2 \\
& E_i^2E_{i\pm1} - (q+q^{-1})E_iE_{i\pm1}E_i + E_{i\pm1}E_i^2 = 0, \\
& F_i^2F_{i\pm1} - (q+q^{-1})F_iF_{i\pm1}F_i + F_{i\pm1}F_i^2 = 0.
\end{align*}
The Hopf algebra structure on $U_q(g\ell_n)$ is given by
\begin{align*}
& \D(K_i) = K_i\ot K_i, \qquad \D(K_i^{-1}) = K_i^{-1}\ot K_i^{-1}, \\
& \D(E_i) = E_i \ot K_iK_{i+1}^{-1} + 1 \ot E_i, \qquad \D(F_i)=F_i \ot 1 + K_i^{-1}K_{i+1} \ot F_i, \\
& \ve(K_i) = 1, \qquad \ve(E_i) = 0 = \ve(F_i), \\
& S(K_i) = K_i^{-1}, \qquad S(E_i) = -E_iK_i^{-1}K_{i+1}, \qquad S(F_i) = -K_iK_{i+1}^{-1}F_i.
\end{align*}
It then happens that, as was remarked in \cite[Subsect. 6.1.2]{KlimSchm-book}, $U_q(s\ell_n)$ is isomorphic as to the Hopf subalgebra of $U_q(g\ell_n)$ generated by $E_i, F_i, \C{K}_i:=K_iK^{-1}_{i+1}$ for $1 \leq i \leq n-1$. 

Explicitly, $U_q(s\ell_n)$ is the algebra generated by $E_i, F_i, \C{K}_i$, with $1 \leq i \leq n-1$, subject to the relations
\begin{align*}
& \C{K}_i\C{K}_j = \C{K}_j\C{K}_i, \qquad \C{K}_i\C{K}_i^{-1} = \C{K}_i^{-1}\C{K}_i=1, \\
& \C{K}_iE_j\C{K}_i^{-1} = q^{2\d_{i,j}-\d_{(i+1),j}-\d_{i,(j+1)}}E_j, \qquad \C{K}_iF_j\C{K}_i^{-1} = q^{-2\d_{i,j}+\d_{(i+1),j}+\d_{i,(j+1)}}F_j,\\
& E_iF_j-F_jE_i =\d_{ij}\frac{\C{K}_i - \C{K}_i^{-1}}{q-q^{-1}}, \\
& E_iE_j = E_jE_i, \qquad F_iF_j = F_jF_i, \qquad |i-j|\leq 2 \\
& E_i^2E_{i\pm1} - (q+q^{-1})E_iE_{i\pm1}E_i + E_{i\pm1}E_i^2 = 0, \\
& F_i^2F_{i\pm1} - (q+q^{-1})F_iF_{i\pm1}F_i + F_{i\pm1}F_i^2 = 0.
\end{align*}
The Hopf algebra structure of $U_q(s\ell_n)$, accordingly, is given by
\begin{align*}
& \D(\C{K}_i) = \C{K}_i\ot \C{K}_i, \qquad \D(\C{K}_i^{-1}) = \C{K}_i^{-1}\ot \C{K}_i^{-1}, \\
& \D(E_i) = E_i \ot \C{K}_i + 1 \ot E_i, \qquad \D(F_i)=F_i \ot 1 + \C{K}_i^{-1} \ot F_i, \\
& \ve(\C{K}_i) = 1, \qquad \ve(E_i) = 0 = \ve(F_i), \\
& S(\C{K}_i) = \C{K}_i^{-1}, \qquad S(E_i) = -E_i\C{K}_i^{-1}, \qquad S(F_i) = -\C{K}_iF_i.
\end{align*}
On the other extreme, there are the extended Drinfeld-Jimbo algebras. For instance, the quotient of $U_q(g\ell_n)$ by the Hopf ideal generated by $K_1K_2\ldots K_n - 1 \in U_q(g\ell_n)$ is called the \emph{extended Drinfeld-Jimbo algebra} of $s\ell_n$, and is denoted by $U_q^{\text{ext}}(s\ell_n)$. More precisely, $U_q^{\text{ext}}(s\ell_n)$ is the algebra generated by $E_i,F_i$, $1\leq i \leq n-1$, and $\hat{K}_j,\hat{K}_j^{-1}$, $1\leq j \leq n$, subject to
\begin{align*}
& \hat{K}_i\hat{K}_j = \hat{K}_j\hat{K}_i, \qquad \hat{K}_i\hat{K}_i^{-1} = \hat{K}_i^{-1}\hat{K}_i=1, \qquad \hat{K}_1\hat{K}_2\ldots \hat{K}_n = 1, \\
& \hat{K}_iE_{i-1}\hat{K}_i^{-1} = q^{-1}E_{i-1}, \qquad \hat{K}_iE_i\hat{K}_i^{-1} = qE_i, \qquad \hat{K}_iE_j\hat{K}_i^{-1} = E_j, \qquad j\neq i,\,i-1,\\
& \hat{K}_iF_{i-1}\hat{K}_i^{-1} = qF_{i-1}, \qquad \hat{K}_iF_i\hat{K}_i^{-1} = q^{-1}F_i, \qquad \hat{K}_iF_j\hat{K}_i^{-1} = F_j, \qquad j\neq i,\,i-1,\\
& E_iF_j-F_jE_i =\d_{ij}\frac{\hat{K}_i\hat{K}_{i+1}^{-1} - \hat{K}_i^{-1}\hat{K}_{i+1}}{q-q^{-1}}, \\
& E_iE_j = E_jE_i, \qquad F_iF_j = F_jF_i, \qquad |i-j|\leq 2 \\
& E_i^2E_{i\pm1} - (q+q^{-1})E_iE_{i\pm1}E_i + E_{i\pm1}E_i^2 = 0, \\
& F_i^2F_{i\pm1} - (q+q^{-1})F_iF_{i\pm1}F_i + F_{i\pm1}F_i^2 = 0.
\end{align*}
Finally, the Hopf algebra structure on $U_q^{\text{ext}}(s\ell_n)$ is given by
\begin{align*}
& \D(\hat{K}_i) = \hat{K}_i\ot \hat{K}_i, \qquad \D(\hat{K}_i^{-1}) = \hat{K}_i^{-1}\ot \hat{K}_i^{-1}, \\
& \D(E_i) = E_i \ot \hat{K}_i\hat{K}_{i+1}^{-1} + 1 \ot E_i, \qquad \D(F_i)=F_i \ot 1 + \hat{K}_i^{-1}\hat{K}_{i+1} \ot F_i, \\
& \ve(\hat{K}_i) = 1, \qquad \ve(E_i) = 0 = \ve(F_i), \\
& S(\hat{K}_i) = \hat{K}_i^{-1}, \qquad S(E_i) = -E_i\hat{K}_i^{-1}\hat{K}_{i+1}, \qquad S(F_i) = -\hat{K}_i\hat{K}_{i+1}^{-1}F_i.
\end{align*}
Let us remark also that the quantized enveloping algebras $U_q(s\ell_n)$ is a Hopf subalgebra of $U_q^{\text{ext}}(s\ell_n)$, see for instance \cite[Subsect. 8.5.3]{KlimSchm-book}. 

As a last note in this subsection, let us note also that in accordance with the real forms of complex Lie algebras, the Drinfeld-Jimbo algebras admit real forms. Along the lines of \cite[Subsect. 6.1.7]{KlimSchm-book}, in the case $q\in \B{R}$, the \emph{compact real form} of $U_q(s\ell_n)$ is the Hopf $\ast$-algebra denoted by $U_q(su_n)$, with the same generators and relations as those of $U_q(s\ell_n)$ as a Hopf algebra, whose $\ast$-structure is given by
\[
\C{K}_i^\ast = \C{K}_i, \qquad E_i^\ast = \C{K}_iF_i, \qquad F_i^\ast = E_i\C{K}^{-1}_i.
\]

\subsection{The coordinate algebras of quantum groups}\label{subsect-coord-alg-quant-gr}~

Following the notation of \cite[Sect. 9]{KlimSchm-book}, we shall denote by $\C{O}_q(G)$ the coordinate algebra of the quantum group $G_q$.

By \cite[Thm. 9.18]{KlimSchm-book} there are (unique, and by \cite[Corollary 11.23]{KlimSchm-book} nondegenerate) Hopf pairings between $U_q(g\ell_n)$ and $\C{O}_q(GL(n))$, and, $U_q^{\text{ext}}(s\ell_n)$ and $\C{O}_q(SL(n))$ - as well as the pairings between $U_{q^{1/2}}(so_{2n+1})$ and $\C{O}_q(SO(2n+1))$, $U_q^{\text{ext}}(so_{2n})$ and $\C{O}_q(SO(2n))$, and finally $U_q^{\text{ext}}(sp_{2n})$ and $\C{O}_q(Sp(2n))$.

We shall, by a slight abuse of notation, address each of these pairings as a pairing between $U_q(\G{g})$ and $\C{O}_q(G)$. Let now $K \subseteq G$ stands for a maximal compact subgroup, with Lie algebra $\G{k}$. Once again, in accordance with the previous subsections, leaning on this duality we introduce $\C{O}_q(G/K)$ as the subalgebra of $\C{O}_q(G)$ dual to the quotient coalgebra $\C{C}_q:=U_q(\G{g}) \ot_{U_q(\G{k})} \B{C}$, which will stand for either of the coalgebras $U_q(g\ell_n) \ot_{U_q(u_n)} \B{C}$, $U^{\rm ext}_q(so_{2n}) \ot_{U_q(u_n)} \B{C}$, or $U^{\rm ext}_q(sp_{2n}) \ot_{U_q(u_n)} \B{C}$.

In case $G:= SL(2)$ we shall consider the subalgebra $\C{O}_q(SU(2))$ of $\C{O}_q(SL(2))$, given in \cite[Sect. 1]{PodlWoro90}, which is the $\ast$-algebra with generators $\a,\b,\g,\d$, subject to the relations
\begin{align*}
& \a\b = q\b\a, \qquad \a\g=q\g\a, \qquad \b\g=\g\b, \qquad \b\d=q\d\b, \qquad\g\d=q\d\g, \\
& \a\d - q\b\g = 1, \qquad \d\a - q^{-1}\b\g = 1,\\
& \b\a^\ast = q^{-1}\a^\ast\b + q^{-1}(1-q^2)\g^\ast\d, \qquad \g\a^\ast = q\a^\ast\g, \qquad \d\a^\ast = \a^\ast\d, \\
& \g\b^\ast = \b^\ast \g, \qquad \d\b^\ast = q\b^\ast \d - q(1-q^2)\a^\ast\g.
\end{align*}

\subsection{The quantum van Est isomorphism}\label{subsect-quant-vanEst-iso}~

Now, $M$ being a right/left SAYD module over $U_q(\G{g})$, let us consider the (relative) Hopf-cyclic complex $C(U_q(\G{g}),U_q(\G{k}),M)$ of the $U_q(\G{g})$-module coalgebra $\C{C}_q:=U_q(\G{g}) \ot_{U_q(\G{k})} \B{C}$. 

The duality arguments of Subsection \ref{subsect-coord-alg-funct} hold verbatim to conclude that $\C{C}_q:=U_q(\G{g}) \ot_{U_q(\G{k})} \B{C}$ being a (left) $U_q(\G{g})$-module coalgebra, $\C{O}_q(G/K)\subseteq \C{O}_q(G)$ is a (left) $\C{O}_q(G)$-comodule algebra. 

Moreover, as was noted in Proposition \ref{prop-SAYD-dual}, $M$ being right/left SAYD module over $U_q(\G{g})$, $M^\vee$ happens to be a left/right SAYD module over $\C{O}(G)$.

Accordingly, dualizing $C^n(U_q(\G{g}),U_q(\G{k}),M):=M\ot_{U_q(\G{g})}\C{C}_q^{\ot\,n+1}$'s, to the left $\C{O}_q(G)$-comodule algebra $\C{O}_q(G/K)$ and a left/right SAYD module $M^\vee $ over $\C{O}_q(G)$, we may associate the cyclic module
\begin{align*}
&C(\C{O}_q(G/K),M^\vee ) = \bigoplus_{n\geq 0} C_n(\C{O}_q(G/K),M^\vee ), \\
& \hspace{5cm} C_n(\C{O}_q(G/K),M^\vee ) := M^\vee \, \square_{\C{O}(G_q)}\, \C{O}_q(G/K)^{\ot\,n+1} 
\end{align*}
with the face, degeneracy, and the cyclic operators as in \eqref{face-classical-Hopf}, \eqref{degeneracy-classical-Hopf}, \eqref{cyclic-classical-Hopf} respectively. 

We denote the cyclic (resp. periodic cyclic) homology of this cyclic module by $HC_\ast(\C{O}_q(G/K),M^\vee )$ (resp. $HP_\ast(\C{O}_q(G/K),M^\vee )$), and call it the cyclic (resp. periodic cyclic) homology of the $\C{O}_q(G)$-comodule algebra $\C{O}_q(G/K)$, with coefficients in the  SAYD module $M^\vee $ over $\C{O}_q(G)$.

On the other extreme, regarding $M$ as a right/left SAYD contra-module $\C{O}_q(G)$, we have a cocyclic module 
\begin{align*}
&C(\C{O}_q(G/K),M) = \bigoplus_{n\geq 0} C^n(\C{O}_q(G/K),M), \\
& \hspace{5cm} C^n(\C{O}_q(G/K),M) := M\ot_{U_q(\G{g})} \Hom(\C{O}_q(G/K)^{\ot\,n+1}, \B{C})
\end{align*}
with the cofaces, codegeneracies, and the cyclic operator as in \eqref{coface-classical-Hopf}, \eqref{codegeneracy-classical-Hopf}, and \eqref{cocyclic-classical-Hopf}.

We shall denote the cyclic (resp. periodic cyclic) homology of this complex by $HC^\ast(\C{O}_q(G/K),M)$ (resp. $HP^\ast(\C{O}_q(G/K),M)$), and we shall refer to it as the \emph{(periodic) Hopf-cyclic cohomology of the $\C{O}_q(G)$-comodule algebra $\C{O}_q(G/K)$, with coefficients in the SAYD contra-module $M$ over $\C{O}_q(G)$}.

Next, along the lines of \cite[Prop. 3.2]{KhalRang05}, we may set up a map 
\begin{equation}\label{q-van-est}
\psi:C^n_{U_q(\G{g})}(\C{C}_q,M) \to C^n(\C{O}_q(G/K),M)
\end{equation}
via
\begin{align*}
& \psi:M \ot_{U_q(\G{g})} \C{C}_q^{\ot\,(n+1)} \to M\ot_{U_q(\G{g})} \Hom(\C{O}_q(G/K)^{\ot\,n+1}, k) \cong \Hom(M^\vee \,\square_{\C{O}_q(G)}\,\C{O}_q(G/K)^{\ot\,n+1}, \B{C}),\\
& m\ot_{U_q(\G{g})} c^0\odots c^n \mapsto m\ot_{U_q(\G{g})} \tilde{c},
\end{align*}
where for any $\tilde{c}:= c^0\odots c^n$ and any $a_0\odots a_n \in\C{O}_q(G/K)^{\ot\,n+1}$,
\[
\langle m\ot_{U_q(\G{g})} \tilde{c}, a_0\odots a_n\rangle := m\langle a_0,c^0\rangle\ldots \langle a_n,c^n\rangle.
\]

\begin{theorem}\label{thm-q-adic-vanEst-I}
Given any right/left SAYD module $M$ over $U_q(\G{g})$, the map \[ C(U_q(\G{g}),U_q(\G{k}),M) \to C(\C{O}_q(G/K), M) \] determined by \eqref{q-van-est} is an isomorphism of cocyclic modules. Therefore, there is a natural isomorphism $HC^*(U_q(\G{g}), U_q(\G{k}),M) \cong HC^*(\C{O}_q(G/K),M)$ of the corresponding cohomology groups.
\end{theorem}

\begin{proof}
Let us first present the commutation with the cofaces. For $0\leq i \leq n-1$,
\begin{align*}
& d_i(\psi(m\ot_{U_q(\G{g})} c^0\odots c^{n-1})) (a_0\odots a_n) = d_i(m\ot_{U_q(\G{g})} \tilde{c}) (a_0\odots a_n) = \\
& m\langle a_0,c^0\rangle\ldots \langle a_ia_{i+1},c^i\rangle\ldots \langle a_n,c^{n-1}\rangle =\\
& \psi(m\ot_{U_q(\G{g})} c^0\odots c^i\ps{1}\ot c^i\ps{2}\odots c^{n-1})(a_0\odots a_n) =\\
& \psi(d_i(m\ot_{U_q(\G{g})} c^0\odots c^{n-1})))(a_0\odots a_n),
\end{align*}
and for the last coface we have
\begin{align*}
& d_n(\psi(m\ot_{U_q(\G{g})} c^0\odots c^{n-1})) (a_0\odots a_n) = d_n(m\ot_{U_q(\G{g})} \tilde{c}) (a_0\odots a_n) = \\
& (m\lt a_n^{\ps{-1}}) \langle a_n^{\ps{0}}a_0,c^0\rangle\ldots \langle a_{n-1},c^{n-1}\rangle = \\
& \langle a_n^{\ps{-1}}, m\ns{-1}\rangle m\ns{0} \langle a_n^{\ps{0}},c^0\ps{1}\rangle\langle a_0,c^0\ps{2}\rangle\ldots \langle a_{n-1},c^{n-1}\rangle = \\
& m\ns{0} \langle a_n^{\ps{0}},c^0\ps{1}\rangle\langle a_0,c^0\ps{2}\rangle\ldots \langle a_{n-1},c^{n-1}\rangle \langle a_n, m\ns{-1}\cdot c^0\ps{2}\rangle = \\
& \psi(d_n(m\ot_{U_q(\G{g})} c^0\odots c^{n-1})))(a_0\odots a_n).
\end{align*}
As for the codegeneracies, for $0\leq j \leq n$ we have
\begin{align*}
& s_j(\psi(m\ot_{U_q(\G{g})} c^0\odots c^{n+1})) (a_0\odots a_n) =  \\
& s_j(m\ot_{U_q(\G{g})} \tilde{c}) (a_0\odots a_j \ot 1 \ot a_{j+1} \odots a_{n-1}) = \\
&m\langle a_0,c^0\rangle\ldots \langle a_j,c^j\rangle\langle 1,c^{j+1}\rangle\langle a_{j+1},c^{j+2}\rangle\ldots \langle a_n,c^{n+1}\rangle = \\
&m\langle a_0,c^0\rangle\ldots \langle a_j,c^j\rangle\ve(c^{j+1})\langle a_{j+1},c^{j+2}\rangle\ldots \langle a_n,c^{n+1}\rangle = \\ 
& \psi(s_j(m\ot_{U_q(\G{g})} c^0\odots c^{n+1})) (a_0\odots a_n).
\end{align*}
Finally, the commutation with the cyclic operator follows from
\begin{align*}
& t_n(\psi(m\ot_{U_q(\G{g})} c^0\odots c^n)) (a_0\odots a_n) =   ((m\lt a_n^{\ps{-1}})\ot_{U_q(\G{g})} \tilde{c}) (a_n^{\ps{0}} \ot a_0\odots a_{n-1}) = \\
& (m\lt a_n^{\ps{-1}})\langle a_n^{\ps{0}}, c^0\rangle \langle a_0,c^1 \rangle \ldots \langle a_{n-1},c^n \rangle =  \langle a_n^{\ps{-1}}, m\ns{-1} \rangle m\ns{0}\langle a_n^{\ps{0}}, c^0\rangle \langle a_0,c^1 \rangle \ldots \langle a_{n-1},c^n \rangle = \\
& m\ns{0}\langle a_n, m\ns{-1}\cdot c^0\rangle \langle a_0,c^1 \rangle \ldots \langle a_{n-1},c^n \rangle = \psi(t_n(m\ot_{U_q(\G{g})} c^0\odots c^n))(a_0\odots a_n).
\end{align*}
\end{proof}

Dually, we have
\begin{equation}\label{q-van-est-II}
\vp: C_n(\C{O}_q(G/K),M^\vee ) \to C_{n,U_q(\G{g})}(\C{C}_q,M^\vee )
\end{equation}
through
\begin{align*}
& \vp: M^\vee \square_{\C{O}_q(G)}\C{O}_q(G/K)^{\ot\,n+1} \to \Hom_{U_q(\G{g})} (\C{C}_q^{\ot\,(n+1)}, M^\vee ),\\
& f \square_{\C{O}(G_q)}\, a_0\odots a_n \mapsto f \square_{\C{O}_q(G)}\,\tilde{a},
\end{align*}
where $\tilde{a}:= a_0\odots a_n$ so that
\[
\langle f \square_{\C{O}(G_q)}\,\tilde{a}, c^0\odots c^n \rangle := f \langle a_0,c^0\rangle \ldots \langle a_n,c^n\rangle.
\]

\begin{theorem}\label{thm-q-adic-vanEst-II}
Given any right/left SAYD contramodule $M^\vee $ over $\C{O}_q(G)$, the map \[ C(\C{O}_q(G/K),M^\vee ) \to C(U_q(\G{g}),U_q(\G{k}), M^\vee ) \] determined by \eqref{q-van-est-II} is an isomorphism of cyclic modules. Therefore, there is a natural isomorphism $HC_*(U_q(\G{g}), U_q(\G{k}),M^\vee ) \cong HC_*(\C{O}_q(G/K),M^\vee )$ of the corresponding homology groups.
\end{theorem}

\begin{proof}
We shall begin with the face operators. For $0 \leq i \leq n-1$,
\begin{align*}
& \d_i(\vp(f\square_{\C{O}(G_q)}\, a_0\odots a_n))(c^0\odots c^{n-1}) = \d_i(f\square_{\C{O}_q(G)}\, \tilde{a})(c^0\odots c^{n-1}) = \\
& (f\square_{\C{O}_q(G)}\, \tilde{a})(c^0\odots c^i\ps{1} \ot c^i\ps{2} \odots c^{n-1}) =  f \langle a_0,c^0\rangle \ldots \langle a_i,c^i\ps{1}\rangle\langle a_{i+1},c^i\ps{2}\rangle \ldots \langle a_n,c^{n-1}\rangle = \\
& \vp(\d_i(f\square_{\C{O}_q(G)}\, a_0\odots a_n))(c^0\odots c^{n-1}).
\end{align*}
The commutation with the last face operator, on the other hand, follows from
\begin{align*}
& \d_n(\vp(f\square_{\C{O}_q(G)}\, a_0\odots a_n))(c^0\odots c^{n-1})(m) = \d_n(f\square_{\C{O}_q(G)}\, \tilde{a})(c^0\odots c^{n-1}) (m) =\\
& f(m\ns{0})\langle a_0,c^0\ps{2}\rangle \langle a_1,c^1\rangle \ldots \langle a_{n-1},c^{n-1}\rangle \langle a_n,m\ns{-1}\cdot c^0\ps{1}\rangle = \\
& f(m\ns{0})\langle a_0,c^0\ps{2}\rangle \langle a_1,c^1\rangle \ldots \langle a_{n-1},c^{n-1}\rangle \langle  a_n^{\ps{-1}},m\ns{-1}\rangle\langle a_n^{\ps{0}}, c^0\ps{1}\rangle =\\
& (a_n^{\ps{-1}}\rt f)(m) \langle a_n^{\ps{0}}a_0,c^0\rangle \langle a_1,c^1\rangle \ldots \langle a_{n-1},c^{n-1}\rangle = \\
& \vp(\d_n(f\square_{\C{O}(G_q)}\, a_0\odots a_n))(c^0\odots c^{n-1})(m),
\end{align*}
for any $m\in M$. Let us next move to the degeneracies. For $0\leq j \leq n$, we have
\begin{align*}
& \s_j(\vp(f\square_{\C{O}(G_q)}\, a_0\odots a_n))(c^0\odots c^{n+1})(m) = \s_j(f\square_{\C{O}(G_q)}\, \tilde{a})(c^0\odots c^{n+1}) (m) = \\
& f \langle a_0,c^0\rangle \ldots \langle a_j,c^j\rangle  \ve(c^{j+1})\langle a_{j+1},c^{j+2}\rangle \ldots \langle a_n,c^{n+1}\rangle  = \\
& \vp(\s_j(f\square_{\C{O}_q(G)}\, a_0\odots a_n))(c^0\odots c^{n+1}).
\end{align*}
We finally present the commutation with the cyclic operator. To this end, it suffices to observe
\begin{align*}
& \tau_n(\vp(f\square_{\C{O}_q(G)}\, a_0\odots a_n))(c^0\odots c^n)(m) = \tau_n(f\square_{\C{O}_q(G)}\, \tilde{a})(c^0\odots c^n) (m) = \\
& f(m\ns{0}) \langle a_0,c^1\rangle \ldots \langle a_{n-1},c^n\rangle \langle a_n, m\ns{-1}\cdot c^0\rangle =  (a_n^{\ps{-1}}\rt f)(m) \langle a_n^{\ps{0}}, c^0\rangle\langle a_0,c^1\rangle \ldots \langle a_{n-1},c^n\rangle = \\
&  \vp(\tau_n(f\square_{\C{O}_q(G)}\, a_0\odots a_n))(c^0\odots c^n)(m)
\end{align*}
for any $m\in M$.
\end{proof}

\section*{Appendix}

The quantum van Est map we defined in this paper, and the quantum
characteristic map we constructed in~\cite{KaygunSutlu:CharMap} are
in fact two different faces of the same construction which we describe
in this section. We shall hereby consider $k$ a field of characteristic 0.

\subsection{The Janus map}~

Let $C$ be a coalgebra that \emph{acts} on an algebra $A$ through
\[
c(ab) = c_{(1)} (a)c_{(2)}(b) 
\] 
for any $c\in C$, and any $a,b \in A$. There is then a
pairing 
\begin{equation}\label{base-pairing}
  diag_{\Delta}(C^\ast(C)\otimes C^\ast(A)) \to C^\ast(A) 
\end{equation}
of the natural cocylic modules associated with $C$ and $A$. Moreover, $H$ being a Hopf algebra, if $C$ is a $H$-module coalgebra and $A$ is a $H$-module algebra so that
\[
h\cdot c(a) = (h\tr c)(a),
\]
then \eqref{base-pairing} may be lifted to a pairing 
\[
diag_\Delta(C^\ast_H(C,M)\otimes C^\ast_H(A,M)) \to C^\ast(A)
\]
of cocyclic modules, for any SAYD module $M$ over $H$, generalizing the Connes-Moscovici characteristic map~\cite[Theorem 6.2]{Kaygun:Uniqueness}.

Accordingly, we have the following.

\begin{theorem}\label{thm:main1}
  Let $H$ be a Hopf algebra, and $C$ a $H$-module coalgebra which acts on a $H$-module algebra $A$. Then, given any SAYD module $M$ over $H$, there is a pairing
\begin{equation}\label{eq:char-map}
    HC_{H}^p(C,M)\otimes HC_{H}^q(A,M)\to HC^{p+q}(A),
  \end{equation}
where $HC^*(A)$ is the ordinary cyclic cohomology of the algebra $A$.   
\end{theorem}

Let us note that the pairing in Theorem~\ref{thm:main1} may also be constructed using the dual cyclic modules, see \cite{KhalRang02,KhalRang05}. To this end, we substitute the Hopf-cyclic cohomologies with their dual theories, where the dual cyclic cohomology ${}^\circ HC^*(A)$ of $A$ is trivial in positive dimensions, i.e. ${}^\circ HC^0(A)=k$
and ${}^\circ HC^n(A)=0$ for $n\geq 1$, see for instance \cite[Rk. 1]{KhalRang05}. Hence, \eqref{eq:char-map} yields also the following.

\begin{theorem}\label{thm:main2}
Let $H$ be a Hopf algebra, and $C$ a $H$-module coalgebra which acts on a $H$-module algebra $A$. Then, given any SAYD module $M$ over $H$, there is a pairing
  \begin{equation}\label{eq:van-est}
    {}^{\circ}HC_{H}^p(C,M)\otimes {}^\circ HC_{H}^p(A,M)\to k
  \end{equation}
in dual cyclic homologies.
\end{theorem}


\subsection{The van Est pairing}~

In the present subsection we shall illustrate the above formalism in the concrete case of the
quantum groups in Section \ref{sect-q-adic-vanEst}. 

Following the (unique, non-degenerate) dual pairing 
\begin{equation}
   \left<\ \cdot\ ,\ \cdot\ \right>\colon U_q(\G{g})\otimes \C{O}_q(G)\to \B{C}
\end{equation} 
of \cite[Thm. 9.18]{KlimSchm-book}, we may introduce an action
 \begin{equation}\label{action-Uq-on-Oq}
   a\tr \varphi = \varphi_{(1)}\left<a,\varphi_{(2)}\right>
 \end{equation}
 of $U_q(\G{g})$ on $\C{O}_q(G)$.  The action \eqref{action-Uq-on-Oq} makes $\C{O}_q(G)$ a $U_q(\G{g})$-module algebra.

 \begin{corollary}\label{cor:van-est1}
Given any (right/left) SAYD module $M$ over $U_q(\G{g})$, there is a characteristic map in Hopf-cyclic cohomology
   \begin{equation}
     HC^p_{U_q(\G{g})}(U_q(\G{g}),M)\otimes HC^{\C{O}_q(G),\, q}(\C{O}_q(G),M^\vee)\to
     HC^{p+q}(\C{O}_q(G))
   \end{equation}
   and a van Est pairing in the dual cyclic cohomology of the form
   \begin{equation}
     {}^\circ HC_{U_q(\G{g})}^p(U_q(\G{g}),M)\otimes {}^\circ HC^{\C{O}_q(G),\,p}(\C{O}_q(G),M^\vee)\to \B{C}.
   \end{equation}
 \end{corollary}

 \begin{proof}
It follows from Theorem~\ref{thm:main1} above that there is the characteristic map
  \begin{equation}
     HC^p_{U_q(\G{g})}(U_q(\G{g}),M)\otimes HC^q_{U_q(\G{g})}(\C{O}_q(G),M)\to
     HC^{p+q}(\C{O}_q(G))
   \end{equation}
   and from Theorem~\ref{thm:main2} that a \emph{van Est pairing} 
   \begin{equation}
     {}^\circ HC^p_{U_q(\G{g})}(U_q(\G{g}),M)\otimes {}^\circ HC^p_{U_q(\G{g})}(\C{O}_q(G),M)\to \B{C}
   \end{equation}
in dual cyclic cohomology.  

Now, $M$ being a right/left SAYD module over $U_q(\G{g})$, it may be endowed with a SAYD contra-module structure over $\C{O}_q(G)$, see Proposition \ref{prop-SAYD-dual} and Remark \ref{rk-contra} above. On the other hand, the left $U_q(\G{g})$-module algebra structure over $\C{O}_q(G)$ yields a right $\C{O}_q(G)$-comodule structure over it. As a result, we have
\begin{align}\label{identification}
\begin{split}
& C^n_{U_q(\G{g})}(\C{O}_q(G),M) = \Hom_{U_q(\G{g})}(M\ot \C{O}_q(G)^{\ot\,n+1}, \B{C}) \cong \\
& \hspace{2cm} \Hom(\C{O}_q(G)^{\ot\,n+1} \, \square_{\C{O}_q(G)}\, M, \B{C}) = C^{n,\,\C{O}_q(G)}(\C{O}_q(G),M^\vee)
\end{split}
\end{align}
via which the cocyclic complex that computes the Hopf-cyclic cohomology of the $U_q(\G{g})$-module algebra $\C{O}_q(G)$, with coefficients in the right/left SAYD module $M$ over $U_q(\G{g})$, is identified with the cocyclic complex\footnote[1]{Note that this cocyclic module is a slightly different version of the one given at the end of Subsection \ref{subsect-coord-alg-funct}, or at Subsection \ref{subsect-contra-comod-alg}. The structure maps may be pulled from those of $C^n_{U_q(\G{g})}(\C{O}_q(G),M)$ using \eqref{identification}. For the sake of simplicity we shall not present the structure maps of this complex here, yet, by a slight abuse of notation we shall denote it as $C^{\ast,\,\C{O}_q(G)}(\C{O}_q(G),M^\vee)$.} that computes the Hopf-cyclic cohomology of the $\C{O}_q(G)$-module coalgebra $\C{O}_q(G)$, with coefficients in the left/right SAYD contra-module $M$ over $\C{O}_q(G)$. As such,
\[
HC_{U_q(\G{g})}^\ast(\C{O}_q(G),M) = HC^{\ast,\,\C{O}_q(G)}(\C{O}_q(G),M^\vee),
\]
and hence in dual theory
\[
{}^\circ HC_{U_q(\G{g})}^\ast(\C{O}_q(G),M) = {}^\circ
   HC^{\ast,\,\C{O}_q(G)}(\C{O}_q(G),M^\vee).
\] 
Both assertions thus follow.
 \end{proof}


\begin{thebibliography}{10}

\bibitem{BlanWign83}
P.~Blanc and D.~Wigner.
\newblock Homology of {L}ie groups and {P}oincar\'{e} duality.
\newblock {\em Lett. Math. Phys.}, 7(3):259--270, 1983.

\bibitem{BohmBrzeWisb09}
G.~Böhm and T.~Brzezinski and R.~Wisbauer.
\newblock Monads and comonads on module categories.
\newblock {\em J. Algebra}, 322(5):1719--1747, 2009.

\bibitem{Brown-book}
K.~S. Brown.
\newblock {\em Cohomology of groups}, volume~87 of {\em Graduate Texts in
  Mathematics}.
\newblock Springer-Verlag, New York, 1994.
\newblock Corrected reprint of the 1982 original.

\bibitem{Brze11}
T.~Brzeziński.
\newblock Hopf-cyclic homology with contramodule coefficients.
\newblock In {\em Quantum groups and noncommutative spaces}, Aspects Math.,
  E41, pages 1--8. Vieweg + Teubner, Wiesbaden, 2011.

\bibitem{ChariPressley-book}
V.~Chari and A.~Pressley.
\newblock {\em A guide to quantum groups}.
\newblock Cambridge University Press, Cambridge, 1995.
\newblock Corrected reprint of the 1994 original.

\bibitem{CiccGava06}
N.~Ciccoli and P.~Gavarini.
\newblock A quantum duality principle for coisotropic subgroups and Poisson quotients.
\newblock {\em Adv. Math.}, 199(1):104--135, 2006.

\bibitem{Conn85}
A.~Connes.
\newblock Noncommutative differential geometry.
\newblock {\em Inst. Hautes Études Sci. Publ. Math.}, (62):257--360, 1985.

 \bibitem{ConnMosc98}
A.~Connes and H.~Moscovici.
\newblock Hopf algebras, cyclic cohomology and the transverse index theorem,
\newblock {\em Comm. Math. Phys.}, 198(1):199--246, 1998.

\bibitem{ConnMosc}
A.~Connes and H.~Moscovici.
\newblock Background independent geometry and {H}opf cyclic cohomology,
  arXiv:math/0505475, (2005).
  
\bibitem{Dupo76}
J.~L. Dupont.
\newblock Simplicial de {R}ham cohomology and characteristic classes of flat
  bundles.
\newblock {\em Topology}, 15(3):233--245, 1976.

\bibitem{FengTsyg91}
P.~Feng and B.~Tsygan.
\newblock Hochschild and cyclic homology of quantum groups.
\newblock {\em Comm. Math. Phys.}, 140(3):481--521, 1991.

\bibitem{Gava02}
F.~Gavarini.
\newblock The quantum duality principle.
\newblock {\em Ann. Inst. Fourier (Grenoble)}, 52(3):809--834, 2002.

\bibitem{Gava07}
F.~Gavarini.
\newblock The global duality principle.
\newblock {\em J. Reine Angew. Math.}, 612:17--33, 2007.

\bibitem{GavaRaki09}
F.~Gavarini and Z. Raki\'c.
\newblock $F_q[M_2]$, $F_q[{\rm GL}_2]$, and $F_q[{\rm SL}_2]$ as
              quantized hyperalgebras.
\newblock {\em Comm. Algebra.}, 37:95--119, 2009.

\bibitem{HajaKhalRangSomm04-II}
P.~M. Hajac, M.~Khalkhali, B.~Rangipour, and Y.~Sommerhäuser.
\newblock Hopf-cyclic homology and cohomology with coefficients.
\newblock {\em C. R. Math. Acad. Sci. Paris}, 338(9):667--672, 2004.

\bibitem{HajaKhalRangSomm04-I}
P.~M. Hajac, M.~Khalkhali, B.~Rangipour, and Y.~Sommerhäuser.
\newblock Stable anti-{Y}etter-{D}rinfeld modules.
\newblock {\em C. R. Math. Acad. Sci. Paris}, 338(8):587--590, 2004.

\bibitem{Hass14}
M.~Hassanzadeh.
\newblock New coefficients for {H}opf {C}yclic {C}ohomology.
\newblock {\em Comm. Algebra}, 42(12):5287--5298, 2014.

\bibitem{HassKhalShap19}
M.~Hassanzadeh and M.~Khalkhali and I.~Shapiro.
\newblock Monoidal categories, 2-traces, and cyclic cohomology.
\newblock {\em Canad. Math. Bull.}, 62(2): 293--312, 2019.

\bibitem{HochKost62}
G.~Hochschild and B.~Kostant.
\newblock Differential forms and {L}ie algebra cohomology for algebraic linear
  groups.
\newblock {\em Illinois J. Math.}, 6:264--281, 1962.

\bibitem{HochMost62}
G.~Hochschild and G.~D. Mostow.
\newblock Cohomology of {L}ie groups.
\newblock {\em Illinois J. Math.}, 6:367--401, 1962.


\bibitem{KhalRang02}
M.~Khalkhali and B.~Rangipour.
\newblock A new cyclic module for {H}opf algebras.
\newblock {\em $K$-Theory}, 27(2):111--131, 2002.

\bibitem{KhalRang05}
M.~Khalkhali and B.~Rangipour.
\newblock A note on cyclic duality and {H}opf algebras.
\newblock {\em Comm. Algebra}, 33(3):763--773, 2005.

\bibitem{JaraStef06}
P.~Jara and D.~Ştefan.
\newblock Hopf-cyclic homology and relative cyclic homology of {H}opf-{G}alois
  extensions.
\newblock {\em Proc. London Math. Soc. (3)}, 93(1):138--174, 2006.

\bibitem{Kaygun:Uniqueness}
A.~Kaygun.
\newblock Uniqueness of pairings in {H}opf-cyclic cohomology.
\newblock {\em J. K-Theory}, 6(1):1--21, 2010.

\bibitem{KaygunSutlu:CharMap}
A.~ Kaygun and S.~ S{\"u}tl{\"u}.
\newblock A characteristic map for compact quantum groups.
\newblock {\em J. Homotopy Relat. Struct.}, 12(3):549--576, 2017.

\bibitem{KhalRang05}
M.~Khalkhali and B.~Rangipour.
\newblock A note on cyclic duality and {H}opf algebras.
\newblock {\em Comm. Algebra}, 33(3):763--773, 2005.

\bibitem{KlimSchm-book}
A.~Klimyk and K.~Schmüdgen.
\newblock {\em Quantum groups and their representations}.
\newblock Texts and Monographs in Physics. Springer-Verlag, 1997.

\bibitem{Majid-book}
S. ~Majid.
\newblock {\em Foundations of quantum group theory}
\newblock Cambridge University Press, 1995.

\bibitem{McCleary-book}
J.~McCleary.
\newblock {\em A user's guide to spectral sequences}, volume~58 of {\em
  Cambridge Studies in Advanced Mathematics}.
\newblock Cambridge University Press, Cambridge, second edition, 2001.

\bibitem{Milne-book}
J.~S. Milne.
\newblock {\em Algebraic groups}, volume 170 of {\em Cambridge Studies in
  Advanced Mathematics}.
\newblock Cambridge University Press, Cambridge, 2017.
\newblock The theory of group schemes of finite type over a field.

\bibitem{Most55}
G.~D. Mostow.
\newblock Self-adjoint groups.
\newblock {\em Ann. of Math. (2)}, 62:44--55, 1955.

\bibitem{Most61}
G.~D. Mostow.
\newblock Cohomology of topological groups and solvmanifolds.
\newblock {\em Ann. of Math. (2)}, 73:20--48, 1961.

\bibitem{PodlWoro90}
P.~Podles and S. L. Woronowicz.
\newblock Quantum deformation of {L}orentz group.
\newblock {\em Comm. Math. Phys.}, 130(2):381--431, 1990.

\bibitem{Positselski-book}
L.~Positselski.
\newblock Homological algebra of semimodules and semicontramodules.
\newblock In series {\em Mathematics Institute of the Polish
              Academy of Sciences. Mathematical Monographs (New Series)}, volume 70.
\newblock Birkh\"{a}user/Springer Basel AG, Basel, 2010.

\bibitem{Rang11}
B.~Rangipour.
\newblock Cup products in {H}opf cyclic cohomology with coefficients in
  contramodules.
\newblock In {\em Noncommutative geometry and global analysis}, volume 546 of
  {\em Contemp. Math.}, pages 271--282. Amer. Math. Soc., Providence, RI, 2011.

\bibitem{RangSutl-V}
B.~Rangipour and S.~S\"{u}tl\"{u}.
\newblock Topological {H}opf algebras and their {H}opf-cyclic cohomology.
\newblock {\em Comm. Algebra}, 47(4):1490--1515, 2019.

\bibitem{RangSutl-II}
B.~Rangipour and S.~Sütlü.
\newblock Cyclic cohomology of {L}ie algebras.
\newblock {\em Documenta Math.}, 17:483--515, 2012.

\bibitem{KobyShap19}
I.~ Kobyzev and I.~ Shapiro.
\newblock A categorical approach to cyclic cohomology of quasi-{H}opf algebras and 
{H}opf algebroids.
\newblock {\em Appl. Categ. Structures}, 27(1):85--109, 2019.

\bibitem{vanEst53}
W.~T. van Est.
\newblock Group cohomology and {L}ie algebra cohomology in {L}ie groups. {I},
  {II}.
\newblock {\em Nederl. Akad. Wetensch. Proc. Ser. A. {\bf 56} = Indagationes
  Math.}, 15:484--492,493--504, 1953.

\end{thebibliography}

\end{document}